\documentclass[preprint,12pt]{elsarticle}
\usepackage{graphicx}
\usepackage{dsfont}
\usepackage{amsmath,amssymb,amsthm,xcolor,mathabx,float,multirow}
\usepackage{hyperref}
\usepackage{amsbsy}

\usepackage{algorithmic}
\usepackage[ruled,vlined]{algorithm2e}
\usepackage[capitalize,nameinlink,noabbrev]{cleveref}
\usepackage{tikzit}

\tikzstyle{new style 0}=[fill=white, draw=black, shape=circle]
\tikzstyle{new style 1}=[fill=black, draw=black, shape=circle]

\tikzstyle{new edge style 0}=[-]
\tikzstyle{new edge style 1}=[->]

\newcommand{\trans  }{^\top}

\newcommand{\Z}{\mathcal{Z}}
\newcommand{\skw}{\rm Skew}
\newcommand{\sym}{\rm Sym}
\newcommand{\mat}{\mathbb{R}}
\newcommand{\lexvec}{\textrm{\bf vec}_{\,\boxbackslash}}

\DeclareMathOperator{\tr}{tr}
\newtheorem{theorem}{Theorem}[section]
\newtheorem{proposition}[theorem]{Proposition}
\newtheorem{lemma}[theorem]{Lemma}
\newtheorem{Example}[theorem]{Example}
\newtheorem{corollary}[theorem]{Corollary}
\newtheorem{observation}[theorem]{Observation}

\journal{Linear Algebra and Its Applications}

\begin{document}

\begin{frontmatter}

\title{The Strong Singular Value Property for Matrices}


\author[label1]{Caleb Cheung} 
\author[label1]{Bryan Shader}
\affiliation[label1]{organization={Department of Mathematics and Statistics, University of Wyoming},
            addressline={1000 E. University Ave.}, 
            city={Laramie},
            postcode={$82072$}, 
            state={WY},
            country={USA}
            \\
            ccheung@uwyo.edu \quad bshader@uwyo.edu}
\begin{abstract}
A new property, the  strong singular value property, is 
introduced, developed, and utilized to study the problem of which lists of nonnegative real numbers occur as the  singular values of a matrix with a prescribed zero-nonzero pattern. 
\end{abstract}


\end{frontmatter}

\section{Introduction}
Throughout, all matrices are over the reals. For nonnegative integers $m$ and $n$,  $\skw (n)$ denotes the
set of all $n \times n$ real skew-symmetric matrices,  $\sym(n)$ denotes the set of all
$n\times n$ real symmetric matrices, and  $\mat^{m\times n}$ denotes the set of all $m \times n$ real
matrices. We consider $\mat^{m\times n}$ as an inner product space with inner product
between $A,B \in \mat^{m\times n}$ given by $\langle A, B\rangle = \tr(A\trans  B)$. 
Here $\trans$ and $\mbox{tr}$ denote the transpose and trace of a matrix, respectively. For an $m\times n$ matrix $A$ and $k \times \ell$ matrix $B$, the \textit{direct sum} of $A$ and $B$, denoted by  $A \oplus B$, is the $(m + k) \times (n+ \ell)$ block diagonal matrix
\begin{align*}
    A \oplus B = \begin{bmatrix}
        A & O\\
        O & B\\
    \end{bmatrix}.
\end{align*}
Throughout the paper, $O$
denotes a zero matrix, $\mathbf 0$ a column vector of zeros,
$J$ an all-ones matrix, and $I$ an identity matrix of appropriate 
size.  Vectors are in boldfaced.

Given an $m\times n$ matrix $A$ with $m\leq n$, 
the {\it singular values} of $A$ are the nonnegative square roots of  eigenvalues of $AA\trans$.  
We use $\Sigma(A)$ to denote the list of singular values of $A$.
As customary, when $\Sigma(A)= \{ \sigma_1, \ldots, \sigma_m\}$ we assume that  $\sigma_1 \geq \sigma_2 \geq \cdots \geq \sigma_m$. 

A \textit{pattern} is a matrix $P=[p_{ij}]$ each of whose entries belongs to $\{0,1\}$.
A \textit{superpattern} of $P$ is 
a pattern obtained from $P$ by changing some of its zeros to ones. 
The \textit{pattern} of the matrix $A$ is the matrix obtained from $A$ by replacing 
each nonzero entry by a $1$.  The set of all matrices with  pattern $P$ is 
denoted by $\Z(P)$. 
The {\it bigraph} of the $m\times n$ matrix $A=[a_{ij}]$ is denoted by 
$\mathcal{B}(A)$, has vertex set $\{1,2,\ldots, m,1',2', \ldots, n'\}$, and an edge joining vertex $i$ to vertex $j'$ if and only if $a_{ij}\neq 0$.

The \textit{Inverse Singular Value Problem for a pattern} (or ISVP-$P$ for short)
is:
\begin{quote}
Given an $m \times n$ pattern $P$ with $m\leq n$ and a non-increasing list 
$\sigma_1\geq \sigma_2 \geq \cdots \geq \sigma_m\geq 0$ of real numbers, does there 
exist a matrix $A$ with pattern $P$ and  singular values $\sigma_1, \ldots, \sigma_m$?
\end{quote}                                                           
The outline of the remainder of the paper is as follows.
In Section 2, we study the ISVP-$P$ for a variety of lists and patterns. In particular, we show that the ISVP-$P$ need not be well-behaved with respect to inserting new nonzero positions into a pattern. Section 3 introduces the Strong Singular Value Property (SSVP for short) for matrices and gives various constructions of matrices 
with the new property. Section 4 presents fundamental results about the SSVP.  The Superpattern Theorem asserts that if $A$ has the SSVP  and pattern $P$, one can solve the ISVP-$P$ for $\Sigma(A)$ for every superpattern of $P$.  The Bifurcation Theorem  asserts that if $A$ has the SSVP and pattern $P$ and nearby $A$ there are always matrices with singular values with a certain property (e.g., two multiple singular values), then one can find a matrix with that property and pattern $P$. The Matrix Liberation Theorem is a  variant of the Superpattern Theorem that can be used when the matrix does not have the SSVP. The proofs of these results are quite analytic and technical, and are left until Section 6.  In Section 5
we apply the fundamental results to establish several results; e.g., 
a characterization of which $n\times n$ patterns allow every list of $n$ distinct nonnegative real numbers as a list of singular values.

\section{Inverse Singular Value Problems}
We begin with a few instances where much can be said about the  ISVP-$P$.  First, we consider the case that the singular values are all equal. 
\begin{proposition}
Let $P$ be an $m \times n$ pattern with $m\leq n$.
Then there is a matrix with pattern $P$ with all singular values equal to $0$ if and only if $P=O$.
\end{proposition}
\begin{proof}
    As the rank of a matrix equals the number of nonzero singular values of the matrix, if there is a matrix with pattern $P$ having all singular values $0$, then $P=O$. 
    Conversely,  if $P=O$, then $P$ itself has all singular values equal to $0$. 
\end{proof}

A matrix $A$ has {\it row-orthonormal} provided 
$AA\trans=I$. 
Patterns of row-orthonormal matrices have been studied in \cite{CHRSS2002, CS2020, F1964, JWP98, W96}.  The diagonal 
matrix whose diagonal entries are $d_1, \ldots, d_m$ is denoted by $\mbox{diag}(d_1, \ldots, d_m)$.

\begin{theorem}
\label{ortho}
Let $P$ be an $m\times n$ pattern with $m\leq n$.
Then there exists a matrix with pattern $P$ having 
all singular values equal and nonzero if and only if 
there is a row-orthonormal matrix with pattern $P$.
Moreover, in this case,   each list of $m$ positive real numbers is 
the list of singular values of a matrix with pattern $P$.
\end{theorem}

\begin{proof}
First suppose there is $A\in \mathcal{Z}(P)$ 
with all singular values equal to the nonzero value $\ell$.
This implies that  $AA\trans $ is a symmetric matrix with all eigenvalues equal to $\ell^2$, which in turn implies $AA\trans = \ell^2 I$. Hence $\frac{1}{\sqrt{\ell}} A$ is a row-orthonormal matrix with pattern $P$. 
Conversely, if $Q \in \mathcal{Z}(P)$ is row-orthonormal,
then $QQ\trans =I$ and each singular value of $Q$ is $1$.

If $\sigma_1, \ldots, \sigma_m$ is a list of $m$ positive 
reals, then $\mbox{diag}(\sqrt{\sigma_1}, \ldots, \sqrt{\sigma_m})Q$ has singular values $\sigma_1, \ldots, \sigma_m$. This establishes the moreover statement. 
\end{proof}

By \Cref{ortho} the ISVP-$P$ for a pattern $P$ 
that allows a row-orthonormal matrix reduces to identifying the lists of nonnegative reals containing at least one 0 that are the singular values of some matrix with pattern $P$.  We do not address this problem here.

The next result characterizes those $m\times n$ patterns $P$ for which
every list of $m$ nonnegative reals, other than the list of $m$ zeros, is the list of singular values of some matrix with pattern $P$. A matrix or vector is \textit{nowhere zero} provided each of its entries is nonzero. 

\begin{theorem}
    Let $P$ be an $m \times n $ pattern with $m \leq n$. Then the following are equivalent:
    \begin{itemize}
        \item[\rm (a)] $P$ allows a rank 1 realization and a rank $m$ realization.  
        \item[\rm (b)] $P$ has the form
        \[ P = 
        \left[ 
        \begin{array}{r|r}
             J & O  \\
        \end{array}
        \right],
        \]
        where $J$ is the $m \times k$ all 1's matrix for some $k$ with  $k \geq m$ and $O$ $m \times (n-k)$.
        \item[\rm (c)] For each list $\Sigma$ of  $m$ nonnegative reals, other than the list of $m$ zeros, there exists a matrix $A$ with pattern $P$ 
        having singular values $\Sigma$. 
    \end{itemize}
\end{theorem}

\begin{proof}
First assume (a).  Since $P$ allows a rank 1 realization, 
the rows and columns of $P$ can be permuted to a matrix of the form 
\[
\left[ 
\begin{array}{c|c}
J & O \\ \hline
O & O 
\end{array} \right].
\]
Since $P$ allows a rank $m$ realization, $P$ contains $m$ 
ones with no two in the same row or column.  Hence (b) holds.

Next assume (b). 
We claim that for each $k\geq 1$, there is a $k\times k$, nowhere zero 
orthogonal matrix $Q$.  For $k=1$ and $k=2$ we have
\[
[1] \mbox{ and } {1/\sqrt{2}}\left[ \begin{array}{rr}
1 & 1 \\
1 & -1
\end{array} \right]. 
\]
For $k\geq 3$, $I_{k}- \frac{2}{k}J$ works. Now, for $m\leq k$, the first $m$ rows of $Q$ form a nowhere zero $m\times k$ row-orthonormal matrix, and 
statement (c) follows from \Cref{ortho}.

Finally, assume (c). 
Then there is a matrix with pattern $P$ having 
singular values $0,0,\ldots, 0, 1$ and a matrix with pattern $P$ having singular values $1,1, \ldots, 1$.
Hence (a) holds. 
\end{proof}

Inverse problems may not behave well with respect to superpatterns.
Here's an example of a pattern $P$,  and a superpattern $\widehat{P}$
such that there is a matrix $A$ with pattern $P$ but there is no matrix $B$ with 
pattern $\widehat{P}$ having the same singular values as $A$. 
\begin{Example}
\label{badsuper}
\rm
The pattern 
\[
P= \left[ \begin{array} {cc} 
1 & 0 \\
0 & 1 \end{array} 
\right] 
\]
allows a matrix $A$ with $\Sigma (A)= \{ 1,1\}$; take $A=P$.
By \Cref{ortho}, the pattern, $\widehat{P}$, obtained from $P$ by changing its $(1,2)$-entry to a $1$, 
does not allow a matrix with singular values $\{1,1\}$, as no matrix with pattern $\widehat{P}$ is 
orthogonal. \qed
\end{Example}

As a final result for this section, we resolve the ISVP-$P$
for the patterns whose  bigraphs are  paths. Throughout the paper, if $A$ is an $n\times n$ matrix, and $\alpha$
and $\beta$ are subsets of $\{1,\ldots,n\}$, then $A[\alpha,
\beta]$ denotes the submatrix of $A$ whose rows have index in $\alpha$ and whose columns have index in $\beta$.  If $\alpha=\beta$, then we abbreviate $A[\alpha,\beta]$ to $A[\alpha]$.

\begin{theorem}
\label{Path}
Let $P$ be an $n\times (n+1)$ pattern whose underlying bigraph is a path on $2n+1$ vertices.  Then a list $\Sigma$ of $n$ nonnegative real numbers forms the singular values of a matrix with pattern $P$ if and only if the elements are distinct and nonzero.
\end{theorem}

\begin{proof}
Without loss of generality $P$ is the $n\times (n+1)$ pattern with  nonzeros in positions 
$(1,1)$, $(1,2)$, $(2,2)$, $(2,3)$, \ldots, $(n-1,n)$, $(n,n)$ and $(n,n+1)$, and zeros elsewhere.

Let $B$ be a matrix with pattern $P$.  Then $BB\trans $ is 
a symmetric matrix whose graph is the path $1$--$2$--$\cdots$--$n$.
In particular, this implies that each eigenvalue of $B{B\trans }$ is simple.
Hence, each singular value of $B$ is simple. Additionally, as $B[\{1, \ldots, n\}, \{2, \ldots, n+1\}]$ is upper triangular with nonzeros on its diagonal,  the rows of $B$ are linearly independent.
Therefore the inverse singular value problem for $P$  has a solution only if the prescribed singular values 
are distinct and nonzero.

Now let $\sigma_1> \sigma_2>\cdots >\sigma_{n+1}=0$ be $n+1$ real numbers.
Then by Theorem 4.2.1 of \cite{Gladwell}, there exists a nonnegative matrix 
$M=[m_{ij}]$ whose graph is the path on $n+1$ vertices and whose eigenvalues are $\sigma_j^2$ for $j=1,2\ldots, n+1$.
Define the $n\times (n+1)$ matrix $B=[b_{ij}]$ 
by 
\[
b_{ii}= \sqrt{\frac{\det M[\{1,\ldots, i\}]}{\det M[\{1, \ldots, i-1\} ]}}
\]
for $i=1,\ldots, n$ (here we take the determinant of a $0\times 0$ matrix to be $1$), 
$b_{i,i+1}=m_{i,i+1}/b_{i,i}$ for $i=1,2,\ldots n$, and $b_{i,j}=0$ for all other $i$ and $j$. Note the pattern of $B$ is $P$.

We claim that $B\trans B =M$. 
If $|j-i| \geq 2$, then the $(i,j)$-entries of $B\trans B$ and $M$
are both zero. 
Next note that for $i=1,\ldots, n$, the  $(i+1,i)$-entry of $B\trans B $ equals
\[ 
b_{i,i+1}b_{i,i}= m_{i,i+1}/b_{i,i}\cdot b_{i,i}=m_{i,i+1},
\] 
as desired. 
Also note for $i=1,\ldots, n$ the $(i,i)$-entry of $B\trans B $ equals
\begin{eqnarray*}
 b_{i,i}^2+ b_{i-1,i}^2&=&  \frac{\det M[\{1,\ldots, i\}]}{\det M[\{1,\ldots, i-1\}]}
+ \frac{ m_{i,i-1}^2\det M[\{1,\ldots, i-2\}]}{\det M[\{1, \ldots, i-1\} ]}
\\
&=& \frac{1}{\det M[\{1, \ldots, i-1\}] } ( 
m_{i,i} \det M[\{1,\ldots,i-1\} ]
\\
 & & -m_{i,i-1}^2 \det M[\{1, \ldots,i-2 \}]
+ m_{i,i-1}^2 \det M[\{1, \ldots, i-2 \}] )\\
&=& 
 m_{i,i}.
\end{eqnarray*}
 Here we take $b_{0,1}$ to be $0$, and the second equality follows from 
 Laplace expansion of the determinant of $M[\{1, \ldots, i\}]$ along its $i$-th row. 

Finally note that the $(n+1, n+1)$-entry of $B\trans B$ equals
\begin{eqnarray*}
b_{n,n+1}^2 & = &  m_{n+1,n}^2/b_{n,n}^2\\
& = & \frac{ m_{n+1,n}^2 \det M[\{1,\ldots, n-1\}]}{ \det M[\{1, \ldots, n\} ]} \\ 
&= & \frac{ m_{n+1,n+1} \det M[\{1,\ldots, n\}]}{ \det M[\{1,\ldots, n\}]} \\
& = &  m_{n+1,n+1}.
\end{eqnarray*}
The second to last equality follows from the fact that 
\[ 0=\det M=
m_{n+1,n+1} \det M[\{1,\ldots, n\}] - m_{n+1,n}^2 \det M[\{1,\ldots, n-1\}]. 
\] 
Hence $B\trans B\ =M$, and $B\trans$ has singular values $\sigma_i$ for $i=1,2,\ldots, n+1$.  It follows that $B$ has singular values $\sigma_i$
for $i=1,2,\ldots, n$.
\end{proof}

A very similar proof to \Cref{Path} resolves the ISVP-$P$ for patterns $P$ whose bigraph is a path with an even number of vertices. 

\section{The Strong Singular Value Property}
 \Cref{badsuper} shows that the ISVP-$P$ can be poorly behaved with respect to superpatterns.
We now describe a condition on a matrix $A$ that guarantees the existence of a matrix 
with the same singular values as $A$ for each superpattern of $A$. We use $A\circ B$
to denote the Schur (that is, entry-wise product) of two matrices $A$ and $B$ of the same size. 

The $m\times n$ matrix $A$ has the \textit{Strong Singular Value Property} (the {\it SSVP}, for short) provided the zero matrix is the only $m\times n$ matrix $X$ such that 
\begin{itemize}
\item[(a)] $XA\trans $ is symmetric,
\item[(b)]  $A\trans  X$ is symmetric, and 
\item[(c)]  $A \circ X=O$.
\end{itemize}

In this section we give examples of matrices having the  SSVP, establish fundamental properties of SSVP matrices, characterize SSVP matrices of 
certain types and show how the problem of recognizing whether a matrix has the SSVP 
can be reduced to the problem of recognizing whether the columns of an associated 
matrix are linearly independent. In the next section we describe how the SSVP can be used 
in the ISVP-$P$. We first explore the SSVP with some examples. 

\begin{Example}
\label{nowhere0}
\rm
Each nowhere zero matrix has the SSVP.
\\
This is clear, since for a nowhere zero matrix $A$, $A\circ X=O$ implies $X=O$.\qed 
\end{Example}
The next example shows that the definition of the SSVP requires
both statements (a) and (b) in the definition.
\begin{Example}
\rm 
Consider
\[
A= \left[ \begin{array}{cc} 
1 & 0 \\
0 & 2
\end{array}
\right]. 
\]
Each matrix $X$ such that $A \circ X=O$
has the form 
\[
X=\left[
\begin{array}{ccc}
0 &  a \\
b &  0 
\end{array} 
\right].
\]
Direct computation shows that $XA\trans $ is symmetric 
if and only if $b=2a$. Similarly, $A\trans  X$ is symmetric
if and only if $a=2b$.  It follows that $A$ has the SSVP.
Moreover, both equations $b=2a$ and $a=2b$ are required to conclude
$X=O$.
\qed
\end{Example}

\begin{Example}
\label{zerorow}
\rm
No $m \times n$ matrix with $m\leq n$ and a row of zeros  has the SSVP. 

\bigskip\noindent
Let $A$ be such a matrix with $m\leq n$ whose $i$-th row is zero, and let $\bf e_i$ be the column vector with a $1$ in position $i$ and zeros elsewhere. 
As $m-1<n$, there is a nonzero vector
$\mathbf z$ such that $A{\mathbf z}={\mathbf 0}$.  Let $X= {\mathbf e_i}{\mathbf z}\trans $.  Then $XA\trans = {\mathbf e_i}(A{\mathbf z})\trans =O $, 
$A\trans  X = A\trans  {\mathbf e_i}{\mathbf z}\trans = {\mathbf  0}{\mathbf z}\trans =O$,  $A\circ X=O$ and $
X \neq O$.  Thus, $A$ does not have the SSVP.
\qed 
\end{Example}
\begin{proposition}
Let $\mathbf y\trans$ be a $1 \times n$ matrix with $n\geq 1$.
Then $\mathbf y\trans$ has the SSVP if and only if $\mathbf y\neq {\bf 0}$.
\end{proposition}

\begin{proof}
By \Cref{zerorow}, if $\mathbf y\trans$ has the SSVP, then $\mathbf y$ is nonzero.

For the converse, assume that $\mathbf y \neq \mathbf 0$, 
and note that  if $\mathbf x$ is a vector 
such that $\mathbf y \mathbf x\trans$ is symmetric, then $\mathbf x$ is a multiple of $ \mathbf y$, 
hence $\mathbf y\circ \mathbf x=\mathbf 0$, implies $\mathbf x=\mathbf 0$. 
Thus $\mathbf y$ has the SSVP. 
\end{proof}

The following shows that having the SSVP is preserved under any combination of 
transposition, permuting rows, permuting columns, negating some set of rows, and 
negating some set of columns. 

\begin{theorem}
\label{equiv}
Let $A$ be an $m\times n$ matrix.
Then for each of $\rm (a)$--$ \rm (e)$ below, 
$A$ has the SSVP if and only if $B$ does. 
\begin{itemize}
\item[\rm (a)] $B=A\trans $.
\item[\rm (b)] $B=PA$ for some permutation matrix $P$.
\item[\rm (c)] $B=AQ$ for some permutation matrix $Q$.
\item[\rm (d)] $B$ is obtained from $A$ by negating some of its rows.
\item[\rm (e)] $B$ is obtained from $A$ by negating some of its columns.
\end{itemize}
\end{theorem}

\begin{proof}
We note in each case the if statement follows from the only if statement. For example, 
if we have established that if $A$ has the SSVP then $A\trans $ does; the statement
if $A\trans $ has the SSVP, then $A$ has the SSVP comes from applying the first statement 
to $A\trans $ and using $(A\trans )\trans =A$.  

Assume that $A$ has the SSVP. 

\bigskip\noindent
{\bf Proof of (a).} Let $B=A\trans $, and  suppose $B\trans  Y$ is symmetric, $YB\trans $ is symmetric, and $B\circ Y = O$ for some matrix $Y$. Consider $X = Y\trans $. Since $B\trans  Y$ is symmetric,  $AX\trans$ (and hence $X A\trans $) is symmetric.
Similarly, since $Y B\trans $ is symmetric, $X\trans A$
(and hence $A\trans X$) is symmetric. Finally, $B\circ Y = O$ implies  $A\trans  \circ X\trans  = O$, which is true only if $A   \circ X = O$.  Since $A$ has the SSVP, $X=O$. Hence $Y=O$, and 
we conclude that $B$ has the SSVP. 

\bigskip\noindent
{\bf Proof of (b).} Let $P$ be a permutation matrix and $B=PA$, and suppose $B\trans  Y$ is symmetric, $YB\trans $ is symmetric, and $B \circ Y = O$ for some matrix $Y$. Consider $X = P\trans  Y$. Since $B\trans  Y$ is symmetric 
and $A\trans X = B\trans Y$, $A\trans X$ is symmetric. Similarly, since $YB\trans $ is symmetric, $PXA\trans P\trans$ is symmetric.  This says that  $XA\trans$ is permutationally similarly to a symmetric matrix, and hence
$XA\trans$ is symmetric.

Finally, since $B\circ Y = O$, we know that $(PA)\circ Y = O$. Pre-multiplying by $P\trans $ gives $A\circ P\trans  Y = O$. Since $A\circ X = O$ only for $X = O$, $P\trans  Y  = O$. Thus $Y = O$, and $B$ has the SSVP.

\bigskip\noindent
{\bf Proof of
(c).} Let $B=AQ$, where $Q$ is a permutation matrix. By (a), $A\trans $ has the SSVP.
By (b), $Q\trans  A\trans $ has the SSVP. Applying (a) again, we see that 
$AQ=(Q\trans  A\trans )\trans $ has the SSVP.

\bigskip\noindent
{\bf Proof of
(d).} Let $B$ be a matrix obtained by negating a subset of rows of $A$. Then   
$B = CA$ for some diagonal matrix each of whose diagonal entries lie in $\{-1,1\}$. 
Note $C^{-1} = C\trans =C$.  Suppose that $B\trans  Y$ and $YB\trans $ are symmetric and  $B \circ Y = O$ 
for some matrix $Y$. Consider $X = CY$. 
Since $B\trans Y=A\trans C\trans C X=A\trans X$, 
$A\trans X$ is symmetric. 
Since $Y B\trans=CXA\trans C \trans$, 
$XA \trans$ is orthogonally similar to a symmetric matrix.  Thus, $XA\trans$ is symmetric.
Finally, since $B \circ Y = O$ we know that $(CA) \circ Y = O$. Pre-multiplying by $C$ gives $A \circ (CY) = O$. As $CY = X$, $CY=O$. Since $C$ only negates some rows of $Y$, $Y=O$. Hence $B$ has the SSVP.

\bigskip\noindent
{\bf Proof of
(e).}  Let $B$ be a matrix obtained from $A$ by negating a subset of its columns. 
Then $B=AE$ for some diagonal matrix whose diagonal entries lie in $\{-1,1\}$.
 Note that $E\trans =E$.
By (a), $A\trans $ has the SSVP.  By (b), $EA\trans $ has the SSVP.
Finally, by (a) again, $AE=(EA\trans )\trans $ has the SSVP. 
\end{proof}

The following shows that having columns of $\mathbf 0$ do not typically affect whether or not a matrix has the SSVP. 
\begin{theorem}
\label{border}
    Let $A$ be a matrix of the form,
    \[ 
A= \left[ \begin{array}{c|c} 
B & O
\end{array} 
\right], 
\]
where $B$ is a nonzero $m \times n$ matrix and $O$ is an $m \times \ell$ zero matrix with $\ell >  0$. Then  $A$ has the SSVP if and only if the rows of $B$ are linearly independent and  $B$ has the SSVP.
\end{theorem}

\begin{proof}
    First suppose that $B$ does not have the SSVP.  Then there exists a nonzero $m\times n$ matrix 
    $X$ such that $B\circ X=O$, and both $B\trans X$ and $XB\trans$ are symmetric.
    Evidently, the matrix
    \[ \widehat{X} = \left[ \begin{array}{c|c} X & O \end{array}\right]
    \] 
    certifies that $A$ does not have the SSVP. Next suppose that the rows of $B$ are linearly dependent. Then there exists an $m \times \ell $ nonzero matrix $Y$ such that $B\trans Y=O$.  The matrix 
    \[ \widehat{X} = \left[ \begin{array}{c|c} O & Y\end{array}\right]
    \] 
    certifies that $A$ does not have the SSVP.  Therefore, we have shown that if $A$ has the SSVP, then $B$ has the SSVP and the rows of 
    $B$ are linearly independent. 
    
    Conversely, assume that  $B$ has the SSVP and its rows are linearly independent. Let $X$ be an $m \times (n + \ell)$ matrix such that $A\trans  X$ and $XA\trans $ are symmetric and that $X \circ A = O$. By the latter assumption, we can write
        \[ 
X = \left[ \begin{array}{c|c} 
U & V
\end{array} 
\right],
\] where $U$ is an $m\times n$ matrix,  $V$ is $m \times \ell$, and $B \circ U= O$. Note that,
    \[ 
B\trans  X = \left[ \begin{array}{c} 
B\trans   \\
\hline
O
\end{array} 
\right]\left[ \begin{array}{c|c} 
U & V
\end{array} 
\right]
 = 
\left[ \begin{array}{c|c} 
B\trans  U & B\trans  V\\
\hline
O & O
\end{array} 
\right]
\] is symmetric, which implies $B\trans  V = O$, and $B\trans  U = U\trans  B$. Since the rows of  $B$ are linearly independent and $B\trans V = O$, $V=O$.

Additionally, 
 \[ 
 XA\trans  = \left[\begin{array}{c|c}
 U & V
 \end{array}\right] \left[\begin{array}{c}
      B\trans  \\
      \hline
     O
 \end{array} \right] = \left[\begin{array}{c}
 UB\trans 
 \end {array}
 \right],
\]
is symmetric, so $UB\trans  = BU\trans $. However, since $B$ has the SSVP, $U = O$. Hence $X = O$ and $A$ has the SSVP.
\end{proof}

We now characterize the $2\times n$ matrices (with $n\geq 2$) that have the SSVP.
By \Cref{border} it suffices to characterize such matrices having no column of zeros.
We note that using the operations in \Cref{equiv} every $2\times n$ matrix with $n\geq 2$ has the form 
of the matrix $A$ described in the statement of the theorem.

\begin{theorem}
Let $A$ be a $2\times n$ matrix with $n\geq 2$ of the form 
\[
A= \left[ \begin{array}{ccc} 
\mathbf a\trans  & \mathbf c\trans  & \mathbf 0\trans   \\
\mathbf b\trans  & \mathbf 0\trans & \mathbf d\trans   \end{array} \right],\]
where  $ \mathbf a\trans$, $\mathbf b \trans$, $\mathbf c\trans$ and
 $\mathbf d\trans$ are nowhere zero vectors of dimension $1\times n_1$, $1 \times n_1$, $1\times n_2$, $1 \times n_3$ respectively.  Then $A$ has  the  SSVP if and only if one of the following holds:
\begin{itemize}
\item[\rm (a)] $n_1>0$, and $n_2=n_3=0$;
\item[\rm (b)]$n_1>0$ and at least one of $n_2$ and $n_3$ is positive; 
\item[\rm (c)] $n_1=0$, $n_2>0$, $n_3>0$ and $\mathbf d\trans  \mathbf d \neq \mathbf c\trans \mathbf c$
\end{itemize}
\end{theorem} 

\begin{proof}
Let $X$ be a matrix such that $A\circ X=O$.
Then $X$ has the form
\[
\left[ 
\begin{array}{lcc}
\mathbf 0\trans & \mathbf 0\trans  & \mathbf w\trans \\
\mathbf 0 & \mathbf y\trans  & \mathbf 0\trans  
\end{array}
\right],
\]
where $\mathbf y$ is $n_2\times 1$ and $\mathbf w$ is $n_3 \times 1$.  Note that 
\begin{equation}
\label{one} A\trans  X= 
\left[ 
\begin{array}{ccc}
O &  \mathbf  b \mathbf y\trans  &  \mathbf a \mathbf w\trans   \\
O & O &  \mathbf c \mathbf w\trans    \\
O & \mathbf d \mathbf y\trans  & O 
\end{array} 
\right] 
\end{equation}
and 
\begin{equation}
\label{two}
XA\trans= \left[ 
\begin{array}{cc}
0 &  \mathbf w\trans \mathbf  d \\
\mathbf y\trans \mathbf  c &  0 
\end{array}
\right].
\end{equation}

If $n_1>0$ and $n_2=n_3=0$, then $A$ is nowhere zero, and hence, (by \Cref{nowhere0}), has the SSVP.

If $n_1>0$  and at least one of $n_2$ and $n_3$ is positive, then (\ref{one}) implies $X=O$ and $A$ has the SSVP.

If $n_1=0$ and at least one of $n_2$ and $n_3$ is zero, then $A$ has a row of zeros, and hence (by \Cref{zerorow}), $A$ does not have the SSVP.

In the remaining case, $n_1=0$  and both $n_2$ and $n_3$ are positive.  Equations (\ref{one}) and (\ref{two}) imply that $A$ has the SSVP if and only if 
\begin{eqnarray}
\label{three}
\mathbf w \mathbf c\trans & = & \mathbf d \mathbf y\trans  \mbox{ and }\\
\label{four}
\mathbf d\trans \mathbf  w&= & \mathbf y\trans \mathbf c
\end{eqnarray}
implies $X=O$.  Note that (\ref{three}) and (\ref{four}) imply that either both $\mathbf w$ and $ \mathbf y$ are zero, or 
$\mathbf w=s \mathbf d$ and  $\mathbf y= s \mathbf c$  for some nonzero scalar $s$ and $\mathbf d \trans \mathbf d=\mathbf c\trans \mathbf c$. Thus, if $\mathbf d \trans \mathbf d \neq \mathbf c\trans \mathbf c$, then  $X=O$ and $A$ has the SSVP.
Also, if $\mathbf d  \trans \mathbf d = \mathbf c\trans \mathbf c$, then taking $\mathbf y= \mathbf c$ and $\mathbf w= \mathbf d$
certifies that $A$ does not have the SSVP. 
\end{proof}

The following characterizes the diagonal matrices that have the SSVP. 
\begin{theorem}
\label{Diagonal}
The diagonal matrix $D=\mbox{\rm diag}(d_1, \ldots, d_n)$ has the SSVP if and only if the diagonal entries of $D$ have 
distinct absolute values, and each diagonal entry of $D$ is nonzero. 
\end{theorem}

\begin{proof}
We first prove the contrapositive of the forward direction.  
By \Cref{zerorow}, if some $d_i=0$, then $D$ does not have the SSVP. 
Now suppose that  $D$ has two diagonal entries of equal absolute 
value, say $|d_{i}|=|d_{j}|$ with $i \neq j$. For $i \in \{1,\ldots m\} $ and $j \in \{1,\ldots, n\}$.  Let $E_{ij}$ denote the $m\times n$ matrix with a one in position $(i,j)$ and zeros elsewhere.
Let $X= E_{ij} + \epsilon E_{ji}$, where $\epsilon=d_i/d_j$ and $E_{ij}$ is the 
$n\times n$ matrix with a $1$ in position $(i,j)$ and zeros elsewhere. Then $D\circ X=O$, $D\trans X$ is symmetric and $XD\trans$
is symmetric.  Thus, $D$ does not have the SSVP. 

To prove the converse,  assume that  each $d_i$ is nonzero,  and $|d_i| \neq |d_j|$ when $i \neq j$.
Let $X$ be a matrix such that $DX$ is symmetric, $XD$ is symmetric and $D \circ X=O$.
Then $DX=X\trans  D$ and $XD= DX\trans $.
So $D^2X= DX\trans  D= XD^2$.  This means that $D^2$ commutes with $X$. 
As $D^2$ is a diagonal matrix with distinct diagonal entries, $X$ is a polynomial 
in $D^2$. In particular $X$ is a diagonal matrix.   As $D\circ X=O$ and each diagonal 
entry of $D$ is nonzero, each diagonal entry of $X$ is zero.  Hence $X=O$, and $D$ has the 
SSVP. 
\end{proof}

\begin{corollary}
\label{diagbord}
Let $A$ be a matrix of the form 
\[ 
A= \left[ \begin{array}{c|c} 
D & O
\end{array} 
\right], 
\]
where $D$ is a diagonal matrix.
Then $A$ has the SSVP if and only if 
each diagonal entry of $D$ is nonzero, and no two diagonal entries 
of $D$  have the same absolute value.
\end{corollary}

\begin{proof}
This follows from \Cref{border} and \Cref{Diagonal}.
\end{proof}

The \textit{term-rank} of the $m\times n$ matrix $A$ is the largest number of 
nonzero entries of $A$ with no two in the same row and column.
A consequence of the K\"onig-Egervary theorem (see \cite{BrualdiRyser}) is that
if $m\leq n$ then $A$ has term-rank less than $m$ if and only if $A$
has an $r\times s$ zero-submatrix for some positive integers $r$ and 
$s$ with $r+s>n $.

\begin{theorem}
\label{termrank}
If the $m\times n$ matrix $A$ has the $SSVP$, 
then $A$ has term-rank $\min \{m,n\}$.
\end{theorem}

\begin{proof}
We prove the contrapositive.
Without loss of generality we may assume $m\leq n$ and that the term-rank of  $A$ is less than $m$.
Then, $A$ has an $r\times s$ 
zero submatrix for some positive integers with $r+s =n+1$.

By permuting rows and columns of $A$,  we may take 
$A$ to have the form 
\[
\left[ 
\begin{array}{c|c}
A_1 & O \\ \hline 
A_2 & A_3
\end{array}
\right],
\]
where $A_1$ is $r \times (n-s)$ and $A_3$ is $(m-r) \times s$.
Since $r+s=n+1>n$, $A_1$ has more rows than columns, and $A_3$ has 
more columns than rows.  Hence there exists a nonzero vector $\mathbf y$ such that $\mathbf y\trans A_1=\mathbf 0\trans$,
and a nonzero vector $\mathbf x$ such that $A_3 \mathbf x=\mathbf 0$.
Let 
\[ X =
\left[ 
\begin{array}{c|c} O & \mathbf y \mathbf x\trans 
\\ \hline 
O & O 
\end{array}
\right].
\]
Then 
$A \circ X=O$, 
\[
A\trans X = \left[\begin{array}{c|c} O & A_1\trans \mathbf y \mathbf x\trans \\ \hline O &O \end{array}  \right]= O,
\] 
and 
\[
XA\trans = \left[\begin{array}{c|c} O & \mathbf y \mathbf x\trans A_3\trans \\ \hline O &O \end{array}  \right]= O.
\] 
Hence $A$ does not have the SSVP. 
\end{proof}

The next theorem characterizes the direct sums of matrices that have the  SSVP. We utilize the following classic result about Sylvester's matrix equation $AX=XB$ \cite{Gantmacher}. For completeness, we sketch a proof here.

\begin{lemma}
\label{sylvester}
Let $A$ and $B$ be symmetric $m\times m$ and $n\times n$ 
matrices, respectively, with spectral decompositions
$\sum_{i=1}^m \lambda_i \mathbf u_i \mathbf u_i\trans$, 
and 
$\sum_{j=1}^n \mu_j\mathbf v_i \mathbf v_i\trans$, respectively.  Then the $m\times n$ matrix $X$ 
satisfies $AX=XA$ if and only if there exist scalars
$c_{ij}$ such that $X=\sum c_{ij} \mathbf u_i \mathbf v_j\trans$, where the sum is over all $(i,j)$ where 
$\lambda_i=\mu_j$.
\end{lemma}

\begin{proof}
Note that $\mathbf u_i \mathbf v_j\trans$ $(i=1,\ldots, m; j=1,\ldots, n)$ is a basis for $\mathbb{R}^{m\times n}$.
Let $X= \sum_{i=1}^m\sum_{j=1}^n c_{ij} \mathbf u_i \mathbf v_j\trans$.  Then $AX=XB$ if and only if 
\begin{equation}
\label{gantmacher}
\sum_{i=1}^m\sum_{j=1}^n c_{ij}\lambda_i \mathbf u_i \mathbf v_j\trans= \sum_{i=1}^m\sum_{j=1}^n \mu_jc_{ij} \mathbf u_i \mathbf v_j\trans.
\end{equation}
The result follows from equating coefficients of the two sides in (\ref{gantmacher}).
\end{proof}

\begin{theorem}[\bf Direct Sum Theorem]
\label{dsum}
\phantom{ }\\
Let $A$ be an $m\times n$ matrix, $B$ be a $p \times q$ matrix and $M=A\oplus B$ with $m+p \leq n+q$.
Then $M$ has the SSVP if and only if  each of the following holds
\begin{itemize}
\item[\rm (a)] 
$ m \leq n$ and $p\leq q$;
\item[\rm (b)] 
both $A$ and $B$ have the SSVP; 
\item[\rm (c)] $A$ and $B$ have no common nonzero singular value; and
\item[\rm (d)] either both $A$ and $B$
have linearly independent rows, or one of $A$ and $B$ is square and invertible.
\end{itemize}
\end{theorem}

\begin{proof}
First assume that $M$ has the SSVP. 
By \Cref{termrank}, $M$ as term-rank $m+p$; and hence $A$ has term-rank $m$ and $B$ has term-rank $p$. 
In particular $m \leq n$, and $p \leq q$.

To show $A$ has the SSVP, let $X$ be an $m\times n$ matrix such that $A \circ X=O$, $A\trans X$
is symmetric and $XA\trans$ is symmetric. Set 
\[ 
\widehat{X} = \left[ \begin{array}{cc} 
X & O \\
O & O \end{array} \right].\]
Then $M\trans \widehat{X}$ and $\widehat{X} M\trans$ are symmetric and $M\circ \widehat{X}=O$.
Since $M$ has the SSVP, $\widehat{X}=O$ and $X=O$.  Thus, $A$ has the SSVP.
A similar argument shows that $B$ has the SSVP.

 Suppose that $A$ and $B$ have a common  singular value $\sigma$. Then  there exist unit vectors  $\mathbf u$, $\mathbf v$, $\mathbf r$ and $\mathbf s$ 
and $\sigma > 0$
such that $A\trans \mathbf u=\sigma \mathbf v$,  $A \mathbf v= \mathbf \sigma \mathbf u$, $B\trans \mathbf r=\sigma \mathbf s$ and $B \mathbf s=\sigma \mathbf r$.
Let 
\[ 
X= \left[ \begin{array}{cc}
O & \mathbf u \mathbf s\trans \\
\mathbf r \mathbf v\trans & O
\end{array} \right].
\]
Then 
\[ 
M\trans X= \left[ \begin{array}{cc}
O & A\trans \mathbf u  \mathbf s\trans \\
B\trans \mathbf r \mathbf v\trans & O 
\end{array} 
\right] = 
\left[ \begin{array}{cc} 
O & \sigma \mathbf v \mathbf s\trans \\
\sigma  \mathbf s \mathbf v\trans & O 
\end{array} \right] 
\]
is symmetric. Similarly, $XM\trans$ is symmetric. Evidently, $M \circ X=O$.  Since $M$ has the SSVP, 
$X=O$.  This is a contradiction. 
Hence (b) holds.

We claim that either the rows of $A$ are linearly independent or $B$ is square and invertible. 
Let $\mathbf u$ and $\mathbf v$ be vectors such that $\mathbf u\trans A=\mathbf 0\trans$ and $\mathbf v\trans B\trans =\mathbf 0\trans$.
Let 
\[ 
X= \left[ 
\begin{array}{cc}
O & \mathbf u \mathbf v\trans\\
O & O 
\end{array} \right] 
\]
Then $M\trans X=O$, $XM\trans=O$, and $M\circ X=O$.  Since $M$ has the SSVP, $X=O$. 
This implies that $\mathbf u=\mathbf 0$ or $\mathbf v=\mathbf 0$. Hence either the rows of $A$ are linearly independent or the 
columns of $B$ are linearly independent.  Since $p\leq q$, the latter holds if and only if $B$ is square and invertible.
Similarly, either the rows of $B$ are linearly independent or  $A$ is square and invertible. 
Therefore, either both $A$ and $B$ have linearly independent rows, or one of $A$ and $B$ is square and invertible. 

Conversely, assume that (a)--(d) hold.  To show that $M$ has the SSVP, consider a matrix of the form 
\[ 
X= \left[ 
\begin{array}{cc}
R & S \\
T & U 
\end{array} 
\right],
\]
where $M\trans \circ X$ and $X \circ M\trans$ are symmetric, and $M\circ X=O$.
Then $A\trans R$ and $R A\trans$ are symmetric, and $A \circ R=O$. 
Since $A$ has the SSVP, $R=O$.  A similar argument shows that $U=O$.

Additionally, the conditions on $M$ and $X$ imply that 
\[ S\trans A= B\trans T \mbox{ and } B  S\trans = T A\trans.
\]
These imply that 
\[
S\trans (AA\trans) = (B\trans B) S\trans \mbox{ and }  T (A\trans A)= (B B\trans) T.  
\]
Since $A$ and $B$ have no common nonzero singular values, \Cref {sylvester} implies that 
the columns of $S\trans$ are null vectors of $B\trans B$, and the rows of $S\trans$ are left null vectors 
of $AA\trans$.  Similarly, the columns of $T$ are null vectors of $BB\trans$ and the rows of $T$
are left null vectors of $A\trans A$.

If the rows of $A$ are linearly independent and the rows of $B$ are linearly independent, 
then $AA\trans$ and $BB\trans$ are invertible. Hence $S=O$ and $T=O$.
If the rows of $A$ are linearly dependent, then by (d) $B$ is square and invertible, 
hence $T=O$. Furthermore $BS\trans=O$, and hence $S=O$. Similarly, if the rows of $B$ are linearly dependent $S=O$ and $T=O$,
We have now shown $X=O$, and hence $M$ has the SSVP. 
\end{proof} 

We conclude this section by explaining how the problem of determining whether a matrix $A$ has the SSVP reduces to a simple linear algebra problem. 
In particular, to each matrix $A$
we associate a matrix $\phi_A$ with the property that $A$ has the SSVP if and only if the matrix $\phi_A$ has linearly independent columns.

Let $A = \lbrack a_{ij}\rbrack$ be an $m\times n$ matrix, 
and let $X=[x_{ij}]$ be a $m\times n$ matrix whose entries are $mn$ distinct variables.
The  equations  $A\trans X = X\trans A$ and $XA\trans = AX\trans$ are equivalent to the following system of 
$\binom{m}{2} + \binom{n}{2}$ linear equations:
\begin{eqnarray}
\label{systemeqs1}
    (A\trans X - X\trans A)_{ij} &=& 0 \quad (1\leq i<j\leq n) \\
    \label{systemeqs2}
    (XA\trans - AX\trans)_{k \ell} &=& 0  \quad (1\leq k < \ell \leq m).
\end{eqnarray}
The condition $A\circ X=O$ is equivalent to 
the system of equations $x_{ij}=0$ for $(i,j)$ 
where $a_{ij} \neq 0$.
In particular, this reduction to a system of linear homogeneous linear equations implies that the following density 
result for matrices with the SSVP.

\begin{theorem}
\label{open}
Let $A$ be a matrix with the SSVP.
Then each matrix $B$ sufficiently close to $B$ (in the Euclidean metric) has the SSVP. 
\end{theorem} 

\begin{proof}
By the preceding discussion, $A$ has the SSVP if and only if the linear system of homogeneous equations in (\ref{systemeqs1}) and (\ref{systemeqs2}) has no nontrivial solution.  This is equivalent to the columns of the coefficient matrix  $C$ of the system being linearly independent, which in turn is equivalent to the existence of a square submatrix intersecting all columns of  $C$ with nonzero determinant.  The result now follows from the continuity of the determinant. 
\end{proof}

For $i \in \{1,\ldots m\} $ and $j \in \{1,\ldots, n\}$.  Let $E_{ij}$ denote the $m\times n$ matrix with a one in position $(i,j)$ and zeros elsewhere.  Then 
\[ X=\sum_{p=1}^m\sum_{q=1}^n x_{pq}E_{pq}.
\] 
  Observe that 
\begin{itemize}
\label{obs2.8}
\item[\rm (a)]
The coefficient of $x_{p,q}$ in the $(i,j)$-th equation in (\ref{systemeqs1}) 
is the $(i,j)$-entry of $A\trans E_{pq}-E_{qp}A$. 
\item[\rm (b)]  The coefficient of $x_{p,q}$ in the $(k,\ell)$-th equation in (\ref{systemeqs2}) is the 
$(k,\ell)$-entry of $E_{pq}A\trans-A E_{qp}$.
\end{itemize}
 
This information can be encoded into a single matrix as follows.
A {\it  hollow matrix} is a square matrix in which each diagonal entry is $0$.
Every $s \times s$ hollow symmetric matrix $B$
can be encoded into an $\binom{s}{2}$ by $1$ vector $\lexvec{B}$  by stacking the portions of the columns of $B$ below its main diagonal on top of each other. 
For example, if  
\[ 
B= \left[ \begin{array}{cccc}
0 & b_{21} & b_{31} & b_{41} \\
b_{21} & 0 & b_{32} & b_{42}\\
b_{31} & b_{32} & 0 & b_{43} \\
b_{41}& b_{42} & b_{43} & 0 
\end{array} \right],
\]
then 
\[ 
\lexvec{B}= 
\left[ \begin{array}{c}
b_{21}\\
b_{31}\\
b_{41} \\
b_{32} \\
b_{42} \\
b_{43}
\end{array} 
\right] .
\]
Using this notation, we see that the coefficient matrix to the system of $\binom{m}{2}+ \binom{n}{2}$ equations in 
(\ref{systemeqs1}) and (\ref{systemeqs2}) has column corresponding to $x_{pq}$ equal to
\begin{equation}
\renewcommand*{\arraystretch}{2}
\label{vec}
 \left[  \begin{array}{c} 
\lexvec (A\trans E_{pq} - E_{qp}A)
\\ \hline 
\lexvec( E_{pq}A\trans - AE_{qp}).
\end{array} \right] .
\end{equation}

Define $\Psi_A$ to be the $(\binom{m}{2} + \binom{n}{2}) \times (mn)$ matrix 
whose columns are indexed by the pairs $(p,q) \in \{1,\ldots, m\}\times \{1\ldots n\}$, and whose 
$(p,q)$-column is the vector in (\ref{vec}).  If we  let $\bf y$ be the $mn \times 1$ vector 
whose rows are indexed in a similar way by the pairs $(p,q)$ and 
\begin{equation}
\label{equation}    
Y=\sum_{p=1}^m \sum_{q=1}^n y_{pq} E_{pq},
\end{equation}
then 
$\Psi_A \bf y={\mathbf 0} $ if and only if 
if and only if both $A\trans Y$ and $YA\trans$ are symmetric. The \textit{SSVP verification matrix} for $A$ is denoted by $\phi_A$ and is the submatrix of $\Psi_A$ consisting of those columns labeled by $(p,q)$
where $a_{p,q} = 0$.

\begin{theorem}
Let $A$ be an $m\times n$ matrix and let $\phi_A$ be its SSVP verification matrix.
Then $A$ has the SSVP if and only if the columns of $\phi_A$ are linearly independent.
Moreover, when the columns of $\phi_A$ are linearly dependent each nonzero vector  $\mathbf y$ such that 
$\phi_A \mathbf y={ \mathbf 0}$ corresponds to a nonzero matrix $Y$ for which both $A\trans Y$ and $Y A\trans$ are symmetric, and $A \circ Y=O$.  
\end{theorem}

\begin{proof}
Note that there is a correspondence between 
vectors $\mathbf y$ whose rows are indexed by  $(p,q)$ for which  $a_{pq}=0$ and matrices $Y$ such that $A\circ Y=O$; namely, the $Y$ corresponding to $\mathbf y$ is 
$Y=\sum y_{pq} E_{pq}$ where the sum is over the same pairs $p$ and $q$. Additionally, with this correspondence $\phi_A \mathbf y=\mathbf 0 $ if and only if both $A\trans Y$ and $Y A\trans $ are symmetric. The result now follows. 

\end{proof}

\begin{Example}
\rm 
Here we  demonstrate how the verification matrix may be used to certify whether or not a matrix has the SSVP.
    Let 
    \begin{align*}
        A = \begin{bmatrix}
            1 & 1 & 0 & 0\\
            0 & 1 & 1 & 0\\
            0 & 0 & 1 & 1
        \end{bmatrix}
    \end{align*}
and
\begin{align*}
    X = \begin{bmatrix}
        x_{11} &x_{12}&x_{13}&x_{14}\\
        x_{21} & x_{22} & x_{23} & x_{24} \\ 
        x_{31} & x_{32} & x_{33} & x_{34}
    \end{bmatrix}.
\end{align*}
Direct calculation gives that $A\trans X-X\trans A$ equals
{\small
\[
 \left[\begin{array}{cccc}
0 & -x_{11} + x_{12} - x_{21} & x_{13} - x_{21} - x_{31} & x_{14} - x_{31} \\
x_{11} - x_{12} + x_{21} & 0 & x_{13} - x_{22} + x_{23} - x_{32} & x_{14} + x_{24} - x_{32} \\
-x_{13} + x_{21} + x_{31} & -x_{13} + x_{22} - x_{23} + x_{32} & 0 & x_{24} - x_{33} + x_{34} \\
-x_{14} + x_{31} & -x_{14} - x_{24} + x_{32} & -x_{24} + x_{33} - x_{34} & 0
\end{array}\right] 
\]}
and  $XA\trans-AX\trans$ equals 
{\small
\[
\left[\begin{array}{ccc}
0 & x_{12} + x_{13} - x_{21} - x_{22} & x_{13} + x_{14} - x_{31} - x_{32} \\
-x_{12} - x_{13} + x_{21} + x_{22} & 0 & x_{23} + x_{24} - x_{32} - x_{33} \\
-x_{13} - x_{14} + x_{31} + x_{32} & -x_{23} - x_{24} + x_{32} + x_{33} & 0
\end{array}\right].
\]
}
Hence,
{\small \[
\Psi_A= \left[\begin{array}{rrrrrrrrrrrr}
1 & -1 & 0 & 0 & 1 & 0 & 0 & 0 & 0 & 0 & 0 & 0 \\
0 & 0 & -1 & 0 & 1 & 0 & 0 & 0 & 1 & 0 & 0 & 0 \\
0 & 0 & 0 & -1 & 0 & 0 & 0 & 0 & 1 & 0 & 0 & 0 \\
0 & 0 & -1 & 0 & 0 & 1 & -1 & 0 & 0 & 1 & 0 & 0 \\
0 & 0 & 0 & -1 & 0 & 0 & 0 & -1 & 0 & 1 & 0 & 0 \\
0 & 0 & 0 & 0 & 0 & 0 & 0 & -1 & 0 & 0 & 1 & -1 \\
0 & -1 & -1 & 0 & 1 & 1 & 0 & 0 & 0 & 0 & 0 & 0 \\
0 & 0 & -1 & -1 & 0 & 0 & 0 & 0 & 1 & 1 & 0 & 0 \\
0 & 0 & 0 & 0 & 0 & 0 & -1 & -1 & 0 & 1 & 1 & 0
\end{array}\right],
\]}
where the rows of $\Psi_A$  are indexed by $(2,1)$, $(3,1)$, $(4,1)$, $(3,2)$, $(4,2)$ and $(4,3)$ entries of 
$A\trans X-X\trans A$, followed by the $(2,1)$, $(3,1)$ and $(3,2)$ entries of $XA\trans -AX\trans$,
and the columns of $\Psi_A$
are indexed by the entries of $X$ in the order $(1,1)$, $(1,2)$, \ldots, $(3,4)$.
The matrix $\phi_A$ is the submatrix of $\Psi_A$ consisting of those columns indexed by the zero entries of $A$ (that is, the columns indexed by $(1,3)$, $(1,4)$, $(2,1)$, $(2,4)$, $(3,1)$ and $(3,2)$).
Thus, 
{\small \[
\phi_A=
\left[\begin{array}{rrrrrr}
0 & 0 & 1 & 0 & 0 & 0 \\
-1 & 0 & 1 & 0 & 1 & 0 \\
0 & -1 & 0 & 0 & 1 & 0 \\
-1 & 0 & 0 & 0 & 0 & 1 \\
0 & -1 & 0 & -1 & 0 & 1 \\
0 & 0 & 0 & -1 & 0 & 0 \\
-1 & 0 & 1 & 0 & 0 & 0 \\
-1 & -1 & 0 & 0 & 1 & 1 \\
0 & 0 & 0 & -1 & 0 & 1
\end{array}\right] .
\]}
The $6\times 6$ submatrix $\phi_A[\{1,2,3,4,5,7\}, \{1,\ldots, 6\}]$ of $\phi_A$ has nonzero determinant. 
Hence, $\phi_A$ has linearly independent columns, and $A$ has the SSVP. \qed
\end{Example}

\begin{Example}
\rm 
Let
\begin{align*}
    B = \begin{bmatrix}
        1 & 1 & 0 & 0\\
        0 & 1 & 1 & 0\\
        0 & 0 & 0 & 0\\
    \end{bmatrix}.
\end{align*}
By \Cref{zerorow}, $B$ does not have the SSVP.  Here we show this using the verification matrix. For $X=[x_{ij}]$, we have $B\trans X -X\trans B$ equals 
{\small \[
\left[\begin{array}{cccc}
0 & -x_{11} + x_{12} - x_{21} & x_{13} - x_{21} & x_{14} \\
x_{11} - x_{12} + x_{21} & 0 & x_{13} - x_{22} + x_{23} & x_{14} + x_{24} \\
-x_{13} + x_{21} & -x_{13} + x_{22} - x_{23} & 0 & x_{24} \\
-x_{14} & -x_{14} - x_{24} & -x_{24} & 0
\end{array}\right]
\]}
and $BX\trans-XB\trans$ equals 
{\small \[ 
\left[\begin{array}{ccc}
0 & x_{12} + x_{13} - x_{21} - x_{22} & -x_{31} - x_{32} \\
-x_{12} - x_{13} + x_{21} + x_{22} & 0 & -x_{32} - x_{33} \\
x_{31} + x_{32} & x_{32} + x_{33} & 0
\end{array}\right].
\] }
This implies that 
{\small
\[ 
\Psi_B= \left[\begin{array}{rrrrrrrrrrrr}
1 & -1 & 0 & 0 & 1 & 0 & 0 & 0 & 0 & 0 & 0 & 0 \\
0 & 0 & -1 & 0 & 1 & 0 & 0 & 0 & 0 & 0 & 0 & 0 \\
0 & 0 & 0 & -1 & 0 & 0 & 0 & 0 & 0 & 0 & 0 & 0 \\
0 & 0 & -1 & 0 & 0 & 1 & -1 & 0 & 0 & 0 & 0 & 0 \\
0 & 0 & 0 & -1 & 0 & 0 & 0 & -1 & 0 & 0 & 0 & 0 \\
0 & 0 & 0 & 0 & 0 & 0 & 0 & -1 & 0 & 0 & 0 & 0 \\
0 & -1 & -1 & 0 & 1 & 1 & 0 & 0 & 0 & 0 & 0 & 0 \\
0 & 0 & 0 & 0 & 0 & 0 & 0 & 0 & 1 & 1 & 0 & 0 \\
0 & 0 & 0 & 0 & 0 & 0 & 0 & 0 & 0 & 1 & 1 & 0
\end{array}\right] 
\]}
and 
{\small \[ 
\phi_B=\left[\begin{array}{rrrrrrrr}
0 & 0 & 1 & 0 & 0 & 0 & 0 & 0 \\
-1 & 0 & 1 & 0 & 0 & 0 & 0 & 0 \\
0 & -1 & 0 & 0 & 0 & 0 & 0 & 0 \\
-1 & 0 & 0 & -1 & 0 & 0 & 0 & 0 \\
0 & -1 & 0 & 0 & 0 & 0 & 0 & 0 \\
0 & 0 & 0 & 0 & 0 & 0 & 0 & 0 \\
-1 & 0 & 1 & 0 & 0 & 0 & 0 & 0 \\
0 & 0 & 0 & 0 & 1 & 1 & 0 & 0 \\
0 & 0 & 0 & 0 & 0 & 1 & 1 & 0
\end{array}\right].
\]
}

It can be verified that 
\[
    \mathbf v = \left[ \begin{array}{r}
        0\\
        0\\
        0\\
        0\\
        1\\
        -1\\ 
        1\\
        0
    \end{array} \right] 
\]
is  a nonzero vector in the nullspace of $\phi_B$. 
Thus, $B$ does not have the SSVP.

The vector $\mathbf v$ corresponds to the matrix
\[
    Y = \left[ \begin{array}{rrrr}
        0 & 0 & 0 & 0\\
        0 & 0 & 0 & 0\\
        1 & -1 & 1 & 0
\end{array} 
\right]. 
\] 
Direct calculation shows $B\trans Y$ and $YB\trans$ are 
both zero matrices (and hence symmetric), and $B \circ Y = O$. Therefore, $Y$ is a matrix that certifies that $B$ does not have the SSVP. 
\qed
\end{Example}

\section{Implications of the SSVP}

In this section we state four fundamental results implied by the SSVP, 
and show how they can be applied to study inverse singular value problems.  The proofs of these results are quite technical in nature, and are left to \Cref{theory}. Fortunately, once the results are established they give simple-to-use tools. 
We note that these results are analogs of those for other strong properties
such as the strong Arnold property \cite{CdV},
the strong spectral property \cite{BARRETT2015}, 
and the strong inner product property
\cite{CS2020}. 

As we noted in \cref{badsuper},  in some instances the inverse singular value problem for a pattern and a superpattern can be quite different.  The first result shows when a matrix $A$ has the SSVP and pattern $P$, then 
every superpattern of $P$ allows a matrix with the same singular values 
as $A$. 

\begin{theorem}[\bf Superpattern Theorem] 
\label{superpattern}
\phantom{ } \\
    Let $A$ have the SSVP and $P$ be a superpattern of the pattern of $A$. Then there exists a matrix
    with pattern $P$ and singular values $\Sigma(A)$. 
\end{theorem}

One application of the the Superpattern Theorem is the following which generalizes the  results 
in \cite{LDuarte,MonfaredShader} for  eigenvalue setting. 

\begin{corollary}[\bf Distinct Singular Values Theorem]
\phantom{ }\\ 
\label{distinct}
Let $\Sigma$ be a set of $m$ distinct positive real numbers 
and $P$ be an $m\times n$ pattern. 
Then there exists an $m\times n$ matrix $A$ with pattern $P$
and singular values $\Sigma$ if and only if $P$ has term-rank $m$.
\end{corollary}

\begin{proof}
Suppose such a matrix $A$ exists. Then 0 is not a singular value of $A$. 
Hence  the rank of $A$ is $m$. 
Thus $P$ has term-rank $m$. 

Next suppose that $P$ has term-rank $m$.  Without loss of generality 
we may assume that the $(1,1),\ldots, (m,m)$ entries  of $P$ are ones.
By \Cref{diagbord}, the matrix 
\[ 
\left[ 
\begin{array}{r|r}
D& O \end{array} 
\right],
\] 
where $D$ is a diagonal matrix whose diagonal entries are the elements of $\Sigma$ has the SSVP.   \Cref{superpattern} now implies the existence of the desired matrix $A$. 
\end{proof}
The Superpattern Theorem allows one to perturb the entries of a matrix with the SSVP to a matrix with a desired superpattern with the same set of singular values. In a similar way, we can perturb the singular values of a matrix while preserving its pattern.  
\begin{theorem}[\bf The Bifurcation Theorem]
    \label{bifurcation}
\phantom{ } \\ 
Let $A$ be an $m\times n$ matrix with  pattern $P$ 
and singular values $\Sigma(A)$. 
If $A$ has the SSVP, then for each $M \in \mat^{m\times n}$ sufficiently close to $A$, there exists a matrix with the SSVP, 
 pattern $P$ and singular values $\Sigma(M)$.
\end{theorem}

If $A$ is an $m \times n$ matrix with distinct singular values $\sigma_1 \geq \dots \geq \sigma_k$ and corresponding multiplicites $m_1,\dots,m_k$, the \textit{multiplicity list} of $A$ is $ (m_1,\dots,m_k)$. The Bifurcation Theorem allows one to study the set of multiplicity lists allowed by a pattern. The following example gives a nontrivial example of a pattern for which the Bifurcation Theorem allows us to show the existence of matrices with multiplicity lists and the same pattern tha are not accessible by elementary means.

\begin{Example}
\phantom{} \rm \\
Let $Q=I_5-\frac{2}{5}J$, and 
\[ 
A= \left[ \begin{array}{c|c}
Q & \mathbf 0 \\ \hline 
\mathbf e_5^T & 1
\end{array} \right]. 
\]
Evidently, $Q$ is an orthogonal matrix, 
and $A$ has singular values  
$(3+\sqrt{5})/2$, $1$, $1$, $1$, $1$, and $(3-\sqrt{5})/2$.
Hence, the multiplicity list of the singular values of $A$ is $(1,4,1)$. 

One can readily verify that $A$ has the SSVP.
Arbitrarily near the singular value decomposition of $A$,
and hence of $A$, there are matrices whose singular values 
have multiplicity lists in 
$ \{ (1,1,3,1),(1,3,1,1),(1,1,1,2,1),(1,1,2,1,1), \\(1,2,1,1,1), (1,1,1,1,1,1)\}$.
Hence by the Bifurcation Theorem for each of these 
multiplicity lists, there exists a matrix whose pattern is that of $A$ having singular values with the given multiplicity list.  

We note that one cannot easily get these multiplicity lists by just scaling the rows of $A$. To see this, 
let $D=\mbox{diag}(d_1,\ldots, d_6)$ be  an invertible diagonal matrix.  Then the graph of $(DA)(DA)\trans$
is a star on $6$ vertices.  It is known that each symmetric matrix 
whose graph is a star on $6$ vertices has at most one multiple eigenvalue.  Hence $DA$ has at most one multiple singular value. 
\qed
\end{Example}

In some instances the Superpattern Theorem is of limited use.  For example, suppose one wishes to study the inverse singular value problem for an $n\times n$ pattern whose bigraph $P$ is a $(2n)$-cycle.  Then each proper spanning subgraph $H$ of $G$ is the vertex disjoint union of paths, and hence (by the Direct Sum Theorem) no matrix with bipartite graph $H$ and a multiple singular value has the SSVP.  Hence, the Superpattern Theorem does not imply anything about ISVP-$P$ and multiple singular values.

The following theorem gives conditions for certain superpatterns of  a matrix $A$ without the SSVP to allow a matrix with the same singular values as $A$.
Its analog for eigenvalues and graphs 
is proven in \cite{FHLS}.  

We first need a few definitions. 
Let $A$ be an $m\times n$ matrix, and $S$ be an $m\times n$ superpattern of the pattern of $A$.
We say that $A$ has the {\it SSVP with respect to $S$} provided $Y=O$ is the only $m\times n$ matrix such that 
$A\trans Y$ is symmetric, $YA\trans$
is symmetric, and $S \circ Y=O$. Given an $m\times n$ matrix $A$ 
with singular value list $\Sigma$,
we use 
$\mbox{Tan}^{\Sigma}_A$ to denote the 
collection of matrices of the form 
$AK+LA$ where $K$, respectively $L$,
is an $m\times m$, respectively $n\times n$ skew-symmetric matrix.
In \Cref{theory}, we show that 
if $A$ is $m\times n$ and  has pattern $P$, then $A$ has the SSVP if and only if 
\[ \mbox{Tan}^{\Sigma}_A+ \mbox{Span}(\mathcal{Z}(P))
= \mathbb{R}^{m\times n}.
\] 
Given an $m\times n$ pattern $D=[d_{ij}] $, the {\it pattern of $A=[a_{ij}] $ in the direction of $D$} is the $m\times n$ 
pattern $S=[s_{ij}]$ for which 
$s_{ij}= a_{ij}$ when $a_{ij} \neq 0$ and $s_{ij}= d_{ij}$
otherwise.

\begin{theorem}[\bf Matrix Liberation Theorem]
\phantom{ } \\
Let $A$ be an $m\times n$ matrix with pattern $P$,  $D$ be a matrix in $\mbox{\rm Tan}^{\Sigma}_A$ 
and  $S$ be the pattern of $A$ in the direction of $D$.
If $A$ has the SSVP with respect to $S$, then there is a matrix $B$ with pattern 
$S$ that has the same singular values as $A$ and has the SSVP.
Moreover, there is such a $B$ with the SSVP.
\end{theorem}

As an initial application of the Matrix Liberation theorem, we completely resolve the inverse singular value problem for the 
pattern
\[ 
C_6= \left[ 
\begin{array}{ccc}
1 & 1 & 0 \\
0 & 1 & 1 \\
1 & 0 & 1 
\end{array}
\right] 
\]
whose bigraph is a $6$-cycle.

\begin{corollary}
There exists a matrix with pattern $C_6$ and singular values $\sigma_1\geq \sigma_2 \geq \sigma_3 \geq 0$ if and only 
if $\sigma_2>0$ and $\sigma_1 \neq  \sigma_3$.
\end{corollary}

\begin{proof}
As $C_6$ is not the pattern of a row-orthogonl matrix,
by \Cref{ortho} there is no matrix with pattern $C_6$
having all singular values equal. As every matrix with pattern $C_6$ has rank at least two, there is no matrix with pattern $C_6$ having $0$ as a singular value of multiplicity 2.  Hence the conditions on the $\sigma$ are necessary. 

\bigskip\noindent
We prove sufficiency by cases.

\bigskip\noindent
{\bf Case 1.} $ \sigma_1>\sigma_2> \sigma_3>0$. 
\\ 
By \Cref{distinct}, there exists a matrix with pattern $C_6$ and these singular values.

\bigskip\noindent
{\bf Case 2.} $ \sigma_1>\sigma_2> \sigma_3=0$. 
\\ 
There exist positive reals $a, b, c, d$ such that 
 $a^2 \neq c^2$, $a^2+ b^2=\sigma_1^2$, $c^2+d^2=\sigma_2^2$, and  
the matrix 
\[ 
N= \left[ \begin{array}{ccc}
a & b & 0 \\
0 & 0 & c \\
0 & 0 & d 
\end{array}
\right] 
\]
has singular values $\sigma_1$, $\sigma_2$ and $0$. 
To verify that  $N$ has the SSVP property with respect 
to the $C_6$,
let 
\[ 
X=\left[\begin{array}{ccc}
0& 0 & x\\ 
y & 0 & 0\\
0 & z & 0
\end{array} 
\right],
\]
and assume that both $N\trans X$ and $X N\trans$ are symmetric. 
Then 
\begin{eqnarray*}
ax&=&cy\\ 
bx&=&dz\\ 
cx&=&ay\\
bz&=&dx.
\end{eqnarray*}
The first and third of these equations imply that $a^2xy=c^2xy$.
Since $a^2\neq c^2$, $xy=0$. Hence $x=0$ or $y=0$. In either case, 
the first of the equations imply that both $x$ and $y$ equal $0$. 
Now the second equation implies that $z=0$.  Hence, $X=O$, and 
$N$ has the SSVP with respect to the desired pattern.

Consider the skew-symmetric matrices
\[
K=\left[ 
\begin{array}{ccc}
0 & (d^2-b^2)c & (a^2-c^2)d \\
(b^2-d^2)c & 0 & 0 \\
(c^2-a^2)d & 0 &0 
\end{array} \right] \mbox{ and } 
\]
\[   
L= 
\left[
\begin{array}{ccc}
            0    &          0  & (b^2 - d^2)a\\
             0   &0 & (c^2 - a^2)b\\
(d^2 - b^2)a & (a^2 - c^2)b &              0
\end{array} \right]. 
\]
One can verify that 
\[ 
KN+NL = 
\left[
\begin{array}{ccc}
0 & 0 & 0\\
0 &  (a^2 + b^2 - c^2 - d^2)bc &  0 \\
-(a^2 + b^2 - c^2 - d^2)ad &  0 & 0
\end{array} 
\right]
\] 
and $KN+NL\in \mbox{Tan}_N^{\Sigma}$.
Hence the $(2,2)$ and $(3,1)$ entries of $KN+NL$ are positive and negative respectively.
By the Matrix Liberation Theorem, there exist a matrix whose  pattern is
$C_6$
having singular values $\sigma_1$, $\sigma_2$ and $0$. 

\bigskip\noindent
{\bf Case 3.} $ \sigma_1=\sigma_2> \sigma_3=0$. 
\\ 
Note that  
\[ 
\frac{1}{\sqrt{3}}\left[ 
\begin{array}{rrr}
1 & 1 & 0 \\
0 & 1 & 1 \\
-1 & 0 & 1
\end{array} 
\right]
\]
has pattern $C_6$, and singular values $\sigma_1$, $\sigma_2$ and $0$.

\bigskip\noindent
{\bf Case 4.} $\sigma_1>\sigma_3>0$, and either 
$\sigma_1=\sigma_2$ or $\sigma_2=\sigma_3$.

Without loss of generality we may assume the desired 
singular values are 
 are $\sigma$, $1$ and $1$,
where $\sigma$ is a positive real number not equal to $1$. 
By \Cref{Path}, there exist positive reals $a$, $b$, $c$ such that 
the matrix 
\[ 
N=\left[ 
\begin{array}{cc}
a& b 
\\
0 & c
\end{array}
\right] 
\] 
has singular values $\sigma$ and $1$. Hence the matrix 
\[ 
M=\left[
\begin{array}{ccc}
a & b & 0 
\\
0 & c & 0 \\
0 & 0 & 1 
\end{array}
\right] 
\]
has singular values $1$, $1$ and $\sigma$.

We can verify that $M$ has the SSVP relative to 
pattern 
\[
P=\left[ \begin{array}{ccc}
1 & 1& 0 \\
0 & 1 & 1 \\
1 & 0 & 1
\end{array} 
\right] 
\]
 by considering a matrix $X$ of the 
form
\[ 
X= \left[ \begin{array}{ccc} 0 & 0 & x \\ y & 0 & 0 \\ 0 & z & 0 
\end{array} \right] 
\]
such that $M^TX$ and $XM^T$ are symmetric.
The symmetry of $M^TX$ is equivalent to $cy=0$, $ax=0$ and $bx=z$.
Since, each of $a$, $b$, and $c$ is nonzero, each of $x$, $y$ and $z$ are zero.  Thus, $X=O$.

Now let 
\[
K= 
\left[ \begin{array}{rrr}
0 & 0 & -1 \\
0 & 0 & 0 \\
1 & 0 & 0 \end{array} 
\right] 
\mbox{ and } L= \left[ \begin{array}{ccc} 
0 & 0 & (1-b^2)/a 
\\
0 & 0 & b \\
(b^2-1)/a & -b & 0 
\end{array} \right].
\]
Then 
\[ 
KM+ML= \left[ 
\begin{array}{ccc}
0 & 0 & 0 \\
0 & 0 & bc\\
(a^2 + b^2 -1)/a & 
0 & 0 
\end{array} 
\right],
\]
and $KM+ML$ is in $\mbox{Tan}^{\Sigma}_A$.

Thus, by the Matrix Liberation Theorem,
there exists a matrix $B$ with singular values $1$, $1$ and $\sigma$ whose pattern is that of $A$ in the direction of $KN+NL$. 
\end{proof}

\section{Applications}
As another consequence of the Matrix Liberation Theorem, we characterize the $m\times m$
patterns having term-rank $m$ and allowing each prescribed collection of distinct singular values, one of which is $0$. We begin with a result patterns that are direct sum, and not of full term-rank.

\begin{lemma}
\label{singsum}
Let $M$ be an $(r+1)\times r$ matrix and $N$ an $s\times (s+1)$ matrix
such that both $M\trans$ and $N$ have the SSVP, each singular value of $M\trans$ 
and of $N$ is positive,  $\Sigma(M) \cap \Sigma(N)= \emptyset$, and the left-nullspace of $M$, respectively the nullspace of $N$, is spanned by a nowhere zero vector. 
Then the following hold:
\begin{enumerate}
    \item[\rm (a)] For each $(r+1) \times (s+1)$  pattern $P$ with at least 
     two nonzero entries, there is a matrix in 
$\mbox{\rm Tan}^{\Sigma}_{M\oplus N}$
     whose pattern, $S$, is that of $M\oplus N$ in the direction of
     \[
     \widehat{P}=
     \left[ \begin{array}{c|c} O & P \\ \hline O & O \end{array} \right]. 
     \] 
    \item[\rm (b)] $M\oplus N$ has the SSVP with respect to $S$.
    \item[\rm (c)] Every superpattern of $S$
    allows a matrix whose singular values are those of $M \oplus N$. 
\end{enumerate}
\end{lemma}

\begin{proof}

Let $K$ be an $(r+1)\times s$ matrix, and $L$ be a
$r \times (s+1)$ matrix.
Then 
\[ 
\renewcommand*{\arraystretch}{1.25}
\begin{array}{l}
\left[ 
\begin{array}{c|c}
O & K\\ \hline 
-K\trans & O \end{array} \right] 
(M\oplus N) +
(M\oplus N) 
\left[ 
\begin{array}{c|c}
O & L\\ \hline 
-L\trans &  O \end{array} 
\right]  \\ 
\quad  =  \left[ 
\begin{array}{c|c}
O & KN+ML \\ \hline
-K\trans M -NL\trans & O 
\end{array} 
\right] 
\end{array}
\]
is in 
$\mbox{Tan}^{\Sigma}_{M\oplus N}$. Consider the linear transformation  
\[ \Theta \colon  \mat^{(r+1)\times s} \times \mat^{r\times (s+1)} \rightarrow \mat^{(r+1)\times (s+1)} \times \mat^{s\times r}\]
 by 
 \[ \Theta((K,L))= (KN+ML, -K\trans M-NL\trans). \]
 The kernel of $\Theta$ consists 
 of matrices $K$ and $L$ such that 
 \begin{eqnarray}
 \label{eqns1}
     KN& = &-ML \\
      \label{eqns2}
     K\trans M & =& -NL\trans .
 \end{eqnarray}
Pre-multiplying (\ref{eqns1}) by $M\trans$ and using $(\ref{eqns2})$
gives 
$M\trans ML= L N\trans N$.
Since $M\trans M$ and $N\trans N$ have no common eigenvalues, 
\Cref{sylvester} implies $L=O$.
Hence by $(\ref{eqns1})$, and the fact $0$ is not a singular value of $N$, 
$K=O$. Hence $\Theta$ is one-to-one, and the image
of $\Theta$ is a subspace of $\mat^{(r+1)\times (s+1)} \times \mat^{s\times r}$
with co-dimension 1.
It follows that 
$\mbox{Tan}^{\Sigma}_{M\oplus N}$
contains a matrix with pattern  that of $M\oplus N$ in the direction of $\widehat{P}$.
This establishes (a). 

We claim that $M\oplus N$ has the SSVP with respect to $S$.
To prove this, let 
\[ 
Y= \left[ \begin{array}{c|c} 
U & V \\ \hline 
W& X 
\end{array} \right] 
\]
be a matrix such that $S \circ Y=O$, 
$(M\trans \oplus N\trans) Y$ is symmetric and $Y(M\trans \oplus N \trans)$
is symmetric.   Then $N \circ X=O$, and both $N\trans X$ and $X N\trans$ 
are symmetric. As $N$ has the SSVP, $X=O$.  Similarly, $U=O$.  Additionally, 
we have the following equations.
\begin{eqnarray}
\label{eqns3}
VN\trans & = & M W\trans \\
\label{eqns4}
M\trans V&=&W\trans N.
\end{eqnarray}
Pre-multiplying (\ref{eqns4}) by $M$ and using 
$(\ref{eqns3})$, we get $MM\trans V=VNN\trans$.
Thus by the hypothesis and \Cref{sylvester},
$V$ is a scalar multiple of a nowhere zero matrix.
Since $Y \circ S=O$, $V=O$. Equation (\ref{eqns3})
and the fact that the columns of $M$ are linearly independent imply that $W=O$, 
Post-multiplying (\ref{eqns4}) by $N\trans$ and   Hence $Y=O$. This establishes (b).

Statement (c) now follows from (a), (b) and the Superpattern Theorem. 
\end{proof}

\begin{corollary}
\label{cycle}
Let $P$ be the pattern whose bigraph is a $(2n)$-cycle with $n\geq 1$, 
and let $0=\sigma_1< \cdots < \sigma_n$ be $n$ distinct real numbers including $0$. 
Then $P$ allows a matrix having the SSVP and singular values $\{\sigma_1, \ldots, \sigma_n\}$.
\end{corollary} 

\begin{proof}
Without loss of generality we may assume the ones of $P$ occur
in positions $(1,1), (2,1), (2,2), \ldots, (n,n-1), (n,n)$, and $(1,n)$. 

If $n=1$, 
then 
\[ 
\frac{\sigma_2}{2}\left[ \begin{array}{rr}
1 & 1 \\
1 & 1 \end{array} \right] 
\] 
is a matrix with pattern $P$ that clearly has the SSVP
and singular values $0$ and $\sigma_2$. 

Now assume $n\geq 2$.
By replacing the $(n-1,n-1)$ and $(1,n)$ entries by zeros, we obtain 
a matrix of the form 
\[ 
\left[
\begin{array}{cc} 
M & O \\ 
O & N 
\end{array} 
\right],
\] 
where $M$ is $(n-1) \times (n-2)$  and has bigraph a path on $2n-3$ vertices, and 
$N= [ 1 \; 1] $.

By \Cref{Path} there exists a matrix $\hat{M}$ with pattern $M$ and singular values 
$\{0, \sigma_2, \ldots, \sigma_{n-1} \}$ and a matrix $\widehat{N}$ 
with pattern $N$ and singular value $ \sigma_n$. 
The result now follows from the \Cref{singsum}.
\end{proof}

Let $P=[p_{ij}]$ be an $n\times n$ pattern with each diagonal entry equal to $1$. The 
{\it digraph of $P$} is the directed graph with vertices $1$, $2$, \ldots, $n$ 
and an arc from $i$ to $j$ 
if and only if $i\neq j$ and 
$p_{ij} \neq 0$.  Thus the digraph of $P$ has no loops. 
\begin{theorem}
Let $P$ by an $m\times m$ pattern having term-rank $m$, 
and assume without loss of generality that $p_{1,1}, \ldots, p_{m,m}=1$.
Let $0=\sigma_1< \sigma_2 < \dots< \sigma_m$. 
Then $P$ allows a matrix with singular values $\{\sigma_1, \ldots, \sigma_m\}$
if and only the digraph of $P$ contains a cycle of length two or more.
\end{theorem}

\begin{proof}
If the digraph of $P$ contains a cycle of length two or more, then by \Cref{dsum}, \Cref{cycle} and the  Superpattern Theorem, $P$ allows a matrix with the desired singular values. 

Otherwise, the digraph of $P$ contains no cycle, and hence $P$ is permutationally 
equivalent to a upper triangular matrix with ones on the diagonal.  Hence every 
matrix with pattern $P$ is invertible,  no matrix in 
$\mathcal{Z}(P)$ 
allows zero as a singular value. 
\end{proof}

\section{Theoretical Underpinnings}
\label{theory}
In this section, we establish the theoretical underpinnings of 
the strong singular value property. The results, like the 
analogous ones for other strong properties that have been established and studied, utilize the  Inverse Function Theorem (IVT) (see 
Theorem 3.3.2 of \cite{krantz2013}).  Throughout, we use 
$\| \cdot \|$ for the standard Euclidean  norm on $\mathbb{R}^n$.

We begin with some linear algebraic results.
Recall we view  $\mat^{m\times n}$ as an inner product space with the inner product $\langle M, N\rangle=\mbox{tr}(M\trans N)$.  We use $S^\perp$ to denote the \textit{orthogonal complement} of a subset $
S$ of $\mat^{m\times n}$ with respect to this inner product.  Thus $\mbox{Sym}(n)^{\perp}= \mbox{Skew}(n)$, and $\mbox{Skew}(n)^{\perp}= 
\mbox{Sym}(n)$.

It is known that if $K$ is a skew-symmetric matrix, 
then $e^{K}=I+K + K^2/2! + \cdots$
is an orthogonal matrix.   
Given an $m\times n$ symmetric matrix $A$ with pattern $P$, 
we will need to know the 
orthogonal complements of the 
spaces 
$\{ KA: K \in \mbox{Skew}(m)\}$, 
$\{ AL: K \in \mbox{Skew}(m)\}$, 
and $\mbox{Span} (\mathcal{Z}(P))$.

\begin{lemma}
\label{perp}
Let $A\in \mathbb{R}^{m\times n}$ with pattern $P=[p_{ij}]$.
Then 
\begin{itemize}
\item[\rm (a)] $\{ KA: K \in \mbox{\rm Skew}(m)\}^{\perp}= \{X \in \mathbb{R}^{m\times n}: 
XA\trans  \mbox{ is symmetric} \}.$ 
\item[\rm (b)] $\{ AL: L \in \mbox{\rm Skew}(n)\}^{\perp}= \{X \in \mathbb{R}^{m\times n}:  
A\trans X \mbox{ is symmetric}\}$.
\item[\rm (c)]  $\mbox{\rm Span}( \mathcal{Z}(P))^{\perp} = \{ X \in \mathbb{R}^{m\times n}: A \circ X=O\}$.
\item[\rm (d)] $A$ has the SSVP if and only if 
 \[ \{ KA: K \in \mbox{\rm Skew}(m)\} + \{ AL: L \in \mbox{\rm Skew}(n)\} + \mbox{\rm Span}( \mathcal{Z}(P))=
 \mathbb{R}^{m\times n}.\]
\item[\rm (e)] $A$ has the SSVP if and only if $\mbox{\rm Tan}_A^{\Sigma}=\mathbb{R}^{m\times n}$.
\end{itemize}
\end{lemma}

\begin{proof}
Note that $X \in \{ KA: K \in \mbox{Skew}(m)\}^{\perp}$ if and only if $\mbox{tr}(X\trans KA)=0$
for all $K \in \mbox{Skew}(m)$.  As 
\[ \mbox{tr}(X\trans K A)=\mbox{tr}(A X\trans K)= \mbox{tr}((X A\trans)\trans K)= 
\langle XA\trans, K \rangle,\] 
we see that $X \in \{ KA: K \in \mbox{Skew}(m)\}^{\perp}$ if and only if $XA\trans$
 is symmetric. Hence (a) holds.  Statement (b) follows by a similar argument. Statement (c) follows from noting that  $X=[x_{ij}] \in \mathcal{Z}(P)^{\perp}$ if and only if $x_{ij}=0$ whenever 
$p_{ij} \neq 0$. By the definition of SSVP,  $A$ has the SSVP if and only if 
\begin{equation}
\label{sum}
 \{ X: A\trans X \mbox{ is symmetric}\} 
 \cap \{ X: XA\trans \mbox{ is symmetric}\} \cap \{ AL: \mbox{Span}( \mathcal{Z}(P))^{\perp}= \{ O\}.
\end{equation}
Statement (d) follows from (a)-(c) and taking orthogonal complements of both sides of (\ref{sum}).
Statement (e) follows from (d) and the definition of $\mbox{Tan}_A^{\Sigma}$. 
\end{proof}

We next discuss  needed concepts and results about multivariable real functions. 
Let $\mathcal{U}$ be an open subset of $\mathbb{R}^n$. The
function 
$f\colon \mathcal{U} \rightarrow \mathbb{R}^m$ is {\it differentiable} at the 
point $\mathbf u \in \mathcal{U}$ provided there exists a linear transformation 
$D_{\mathbf u}f: \mathbb{R}^n \rightarrow \mathbb{R}^m$
such that 
\[ 
\lim_{\mathbf x \rightarrow \mathbf u}
\frac{\| f(\mathbf x) -f(\mathbf u) -D_{\mathbf u} f(\mathbf x - \mathbf u) \|} 
{ \| \mathbf x - \mathbf u \|}=0.  
\]
It is known that if $D_{\mathbf u}f$ exists, then it is unique. 
In this case, $D_{\mathbf u}f$ is called the {\it total derivative of $f$ at $\mathbf u$}.  The map $\mathcal{L}: 
\mathbb{R}^n \rightarrow \mathbb{R}^m$ by $\mathcal{L}(\mathbf x)=
f(a)+ D_{\mathbf u}f(x-a)$ is called the {\it linearization of $f$ at $\mathbf u$}.

The function $f$ determines functions 
$f_i\colon \mathcal{U} \rightarrow \mathbb{R}$ $(i=1,2, \ldots, m)$ 
with 
\[ f(\mathbf x)=(f_1(\mathbf x), \ldots, f_m(\mathbf x))\]
for all $\mathbf x=(x_1, \ldots, x_n)\in \mathcal{U}$. 
It is known that if  $\frac{\partial{f_i}}{\partial x_j}$ exists and is continuous for each $i$ and $j$ and each $\mathbf x$ in $\mathcal{U}$, then 
$D_{\mathbf u}f$ exists, and is given by 
\[ 
D_{\mathbf u}f (\mathbf  x) = \mbox{Jac}_f\;{\mathbf x}, \]
where $\mbox{Jac}_f$ is the $m\times n$ matrix whose $(i,j)$-entry equals
\[ 
\frac{\partial{f_i}}{\partial x_j} (\mathbf u)
\]
for $i=1,\ldots, m$; and $j=1,\ldots, n$.
The matrix $\mbox{Jac}_f$ is known as the {\it Jacobian matrix of $f$} at $\mathbf x= \mathbf u$.
The relationship between $D_{\mathbf u}$ and $\mbox{Jac}_f$ and basic facts about linear algebra give the following facts. 
\begin{observation}
The following are equivalent:
\begin{itemize}
\item[\rm (a)] $D_{\mathbf u} f$ is an onto map.
\item[\rm (b)] The column space of $\mbox{Jac}_f$ is  $\mathbb{R}^n$.
\item[\rm (c)] The rows of $\mbox{Jac}_f$ are linearly dependent. 
\end{itemize}
\end{observation}

We are now ready to state the Inverse Function Theorem which relates the 
local invertibility of a function to the invertibility of its derivative (see \cite{krantz2013}).  

\begin{theorem}[\bf Inverse function theorem (IVT)] 
\label{IVT}
\phantom{} \\
Let $\mathcal{U}$ be an open set of $\mathbb{R}^m$,  $\mathbf u \in \mathcal{U}$
and  $f\colon \mathcal{U} \rightarrow \mathbb{R}^m$ be a smooth function on $\mathcal{U}$.  If the Jacobian matrix of $f$ at $\mathbf u$ is invertible
(equivalently, $D_{\mathbf u}f$ is an onto function),  then there exists an open neighborhood $\mathcal{V} \subseteq \mathcal{U}$
of $\mathbf u$ such that the restriction $f|{\mathcal{V}}$ is injective; $f(
\mathcal{V})$
is open, and the inverse of $f|\mathcal{V}$ is differentiable. \end{theorem}

The following consequence of the IVT  is convenient for our purposes.

\begin{corollary}
\label{cor3}
Let $\mathcal{U}$ be an open set of $\mathbb{R}^n$, $\mathbf u\in \mathcal{U}$ and 
$f\colon \mathcal{U} \rightarrow \mathbb{R}^m$ be a smooth function on $\mathcal{U}$.
If the Jacobian matrix of $f$ at $\mathbf x=\mathbf u$ 
has rank $m$ (equivalently, $D_{\mathbf u}f$ is onto),  then there exists an open neighborhood 
 $\mathcal{V} \subseteq \mathcal{U}$ of $\mathbf u$ and an open neighborhood $\mathcal{W}\subseteq \mathbb{R}^m$ of $f(\mathbf u)$
such that $\mathcal{W}$ is contained in $f(\mathcal{V})$.
\end{corollary}

\begin{proof}
Assume $\mathbf u \in \mathcal{U}$ and the  Jacobian of 
 $f$ at $\mathbf u$ has rank $m$.
Then $n\geq m$ and  there exist $m$ variables, 
which we may assume  to be  $x_1, x_2, \ldots, x_m$, 
such that $\mbox{Jac}_f[\{1,\ldots, m\}, \{1,\ldots, m\}]$ is invertible.
The subset $\mathcal{R}=\mathbb{R}^{m \times \{[u_{m_1}, \ldots, u_{n}]\trans\}}$ is homeomorphic 
to $\mathbb{R}^m$, and we can naturally restrict $f$ to $\mathcal{R}$.
We focus on  the restriction of $f$ to $\mathcal{R}$. 
Its Jacobian matrix is an invertible matrix. 

By \Cref{IVT}, there is an open neighborhood $\hat{\mathcal{V}}$ of $\mathcal{R}\cap \mathcal{U}$
containing $\mathbf u$
and an open neighborhood $\mathcal{W}$  such that $\mathcal{W} \subseteq f(\hat{\mathcal{V}})$. 
Taking   $\mathcal{V}$ to be any open set of $\mathcal{U}$ containing $\hat{\mathcal{V}}$ completes the proof. 
\end{proof}

\bigskip\noindent
{\bf Proof of the Superpattern Theorem.}
Let $A$ be an $m\times n$ matrix with the SSVP and let $P$ be a superpattern of the pattern of $A$. Our goal is to find a matrix with pattern $P$ and the same singular values as $A$.  

Define $f\colon \mbox{\skw}(m)\times \mbox{\skw}(n)\times  \mbox{Span} (\mathcal{Z}(P)) \to \mat^{m\times n}$ via
\begin{align*}
    (K,L,B) \mapsto e^{K}Ae^{L} + B.
\end{align*}
We claim that 
$D_{(O,O,O)}f$ is the linear transformation given by 
\[ 
(K,L,B) \mapsto KA + AL + B. 
\]
Note $f(K,L,B)=g(K)Ah(L) +i(B)$,
where $g\colon \skw(m)\rightarrow \mat^{m\times m}$ by  $g(K)=e^K$;
$h\colon \skw(n)\rightarrow \mat^{n\times n}$ by  $h(L)=e^L$;
and $i\colon \mat^{m\times n} \rightarrow \mat^{m\times n}$ by $i(B)=B$.
Since $e^K= I +K + (K^2)/2! + \cdots $,  $D_{O}g(K)=K$.
Similarly, $D_{O}g(L)=L$. Finally, $D_{O}i(B)=B$. 
Hence by the addition and product rules,  
\[ D_{(O,O,O)}f(K,L,B)=
KA +AL +B.\] 
Since $A$ has the SSVP,  \Cref{perp} implies that $D_{(O,O,O)}f$ is onto.

Since $f$ is continuous at $(O,O,O)$, there 
exists  an open neighborhood 
$\mathcal{U} \subseteq \skw (m)\times \skw (n) \times \mbox{Span}(\mathcal{Z}(P))$ of $(O,O,O)$   such that for each $(K,L,B) \in \mathcal{U}$, $A-B$ has the same pattern as $A$. By \Cref{cor3},  there is a open neighborhood $\mathcal{V} \subseteq \mathcal{U}$ of $(O,O,O)$ and an open neighborhood $\mathcal{W} \subseteq \mat^{m\times n}$ of $A=f(O,O,O)$ such that 
\begin{equation}
\label{subseteq}
\mathcal{W} \subseteq f(\mathcal{V}).
\end{equation}
Thus, there exists a matrix $\hat{A} \in \mathcal{W}$ such that $\hat{A}$
has pattern $P$. 
By (\ref{subseteq}), there exists $(K,L,B) \in \mathcal{V}$ such that $f(K,L,B) = \hat{A}$.  Now,
\begin{align*}
    \hat{A} &= e^KAe^L + B\\
    \hat{A}- B &= e^kAe^L.
\end{align*}
By the choice of $\mathcal{U}$,  $A-B$ has the same pattern as $A$. Now, by the choice of $\hat{A}$, the pattern of 
$\hat{A} - B$ is $P$.
Since $K$ and $L$ are skew-symmetric, $e^K$ and $e^L$ are orthogonal. Thus $\hat{A}-B$ has the same singular values as $A$, and $\hat{A}-B$ is the desired matrix.
\qed

\bigskip
\bigskip
We next turn our attention to the Bifurcation Theorem. 
Our goal is to show that if $A$ has the SSVP and pattern $P$, then for each $M$ 
sufficiently close to $A$ there is a matrix $\hat{M}$ with pattern $P$, the SSVP 
and the same singular values as $M$. 

\bigskip \noindent
{\bf Proof of the Bifurcation Theorem.}
Let $A$ be an $m\times n$ matrix with pattern $P$ having the SSVP.  Consider the map $g\colon\skw(m)\times\skw(n) \times \mbox{Span} (\mathcal{Z}(P)) \to \mat^{m\times n}$ via:
    \begin{align*}
        (K,L,B)\mapsto e^K(A+B) e^L.
     \end{align*}
This function shares the same derivative as the function in the previous proof, so we may use \cref{cor3}. There exists be an open neighborhood $\mathcal{U}$ of $(O,O,O)$ such that for each $(K,L,B) \in \mathcal{U}$, $A+B$ has the same pattern as $A$. 
Furthermore,  (by \Cref{open}) $\mathcal{U}$ can be chosen so that each $A+B$
has the SSVP.  By \Cref{cor3} there is an open neighborhood $\mathcal{V}\subseteq \mathcal{U}$ of $(O,O,O)$ and an open neighborhood $\mathcal{W} \subseteq \mat^{m\times n}$ of $A=g(O,O,O)$ such that $\mathcal{W} \subseteq g(\mathcal{V})$. Let $M \in \mathcal{W}$. Then, there is a $(K, L, B) \in \mathcal{V}$ such that
\begin{align*}
    M=g(K,L,B) = e^{K}(A+B)e^{L}.
\end{align*}
Since $e^{K}$ and $e^{L}$ are orthogonal matrices, $A+B$ and $M$ share the same singular values. On the other hand, $B\in \mathcal{V} \subseteq \mathcal{U}$ implies that  $A+B$ and  $A$ have the same pattern.  Therefore, $A+B$ is a matrix in $\mathcal{Z}(P)$ with $\Sigma(M) = \Sigma(A)$. 
\qed 

\bigskip\noindent
{\bf The Matrix Liberation Theorem.}
Finally, we turn our attention to a proof of the Matrix Liberation Theorem. The proof hinges upon a refinement of the inverse function theorem. For $\epsilon>0$, 
$\mathcal{B}({\mathbf u}, \epsilon)$ denotes the open sphere centered at $\mathbf u$
of radius $\epsilon$ in a given Euclidean space.  The following is 
Proposition 2.5.6  of \cite{MRA}.

\begin{proposition}
\label{extendedIVT}
Let $f\colon \mathcal{U} \subseteq E \rightarrow F$ be a function 
between an open set $\mathcal{U}$ of a  finite dimensional Euclidean space $E$ and a finite dimensional Euclidean space $F$, and assume that $f$ is continuously differentiable on $\mathcal{ U}$ and 
 $\mathbf u \in U$.
Then there exist an open set $\mathcal{V}$ of $\mathcal{U}$ containing $\mathbf u$ and  positive constants $c$
and $r$ such that for 
 each ${\mathbf v} \in \mathcal{V}$ and each $p$ with  $r>p>0$ the function
$f$ maps an open subset of $\mathcal{B}( {\mathbf v},cp)$ onto $\mathcal{B}(f({{\mathbf v}}),p)$.
\end{proposition}

\Cref{extendedIVT} implies
the following result, which can be used to prove the Matrix Liberation Theorem for the SSVP (and the other strong properties as well). 

\begin{theorem}
\label{IVTLiberate}
Let $f:\mathbb{R}^n \rightarrow  
\mathbb{R}^m$ be a smooth function
whose derivative  $Df$ at $\mathbf x= \mathbf 0$ exists and  let $\mathbf d= Df(\mathbf u)$ be a nonzero vector in the tangent space of $f$ at $\mathbf x= \mathbf 0$. 
If 
\begin{equation}
\label{cond} 
\mbox{\rm Im}(Df) + \mbox{\rm Span}(\{ \mathbf e_i: i \mbox{ in support of $f(\mathbf 0)$ or the support of $\mathbf d$ })\}  
= \mathbb{R}^n,
\end{equation}
then there exists $\widehat{\mathbf{x}} \in \mathbb{R}^m$
such that the support of $f(\widehat{\mathbf{x}})$ 
is the union of 
the supports of $f(\mathbf 0)$ and $\mathbf d$.  Moreover, $\widehat{\mathbf{x}}$ 
can be chosen so that the 
\begin{equation}
\label{cond'}
\mbox{Im}(Df_{\widehat{\mathbf{x}}} + \mbox{Span}(\{ \mathbf e_i: i \mbox{ in support of $f(\mathbf 0)$ or the support of $\mathbf d$  }\})  
= \mathbb{R}^n.
\end{equation}
\end{theorem} 

\begin{proof}
Assume that (\ref{cond}) holds. 
Without loss of generality, the support of $\mathbf d$ is $\{1,\ldots, k\}$. Additionally, without loss of generality 
we may assume that $\|\mathbf u\|=1$.
Define $g: \mathbb{R}^n \times \mathbb{R}^k \rightarrow \mathbb{R}^m$ by 
\[g(\mathbf x, \mathbf y)= f(\mathbf x) +
\left[ \begin{array}{c} \mathbf y\\ \mathbf 0
\end{array} \right]. \]  Note that 
\[ 
Dg(\mathbf x, \mathbf y)=Df(\mathbf x) + \left[ \begin{array}{c} \mathbf y \\ \mathbf 0 
\end{array} \right].\]  
The hypotheses imply that $D_{(\mathbf 0,\mathbf 0)}g$ is onto.

Let 
$\tau$ be the minimum of $|d_i|$, $|f(\mathbf 0)_j/ d_j|$, $|(f(\mathbf 0))_k|$ over $i$
such that $d_i\neq 0$, $j$ with $d_j\neq 0$ and 
$f(\mathbf 0)_j \neq 0$ and $k$ for which $\mathbf f(\mathbf 0)_k\neq 0$.

\Cref{extendedIVT} implies that there exists an open neighborhood $\mathcal{V}$ of $(\mathbf 0, \mathbf 0)$ and 
constants $r$ and $c$
such that for all $(\mathbf x, \mathbf y) \in \mathcal{V}$,
and all 
$p$ with $r>p>0$, $\mathcal{B}(g(\mathbf x, \mathbf y),p) \subseteq 
g(\mathcal{B}((\mathbf x,\mathbf y), cp))$.

Choose $p= \min \{\frac{\tau}{2c\|\mathbf d\|}, r/2\}$.
Since $D_{\mathbf 0}f$ exists, for $t>0$ sufficiently small, 
\[
\|f(t \mathbf u)-f(\mathbf 0) - t\mathbf d\|=
\|f(t \mathbf u)-f(\mathbf 0) - D_{\mathbf 0}f (t\mathbf u) \|<tcp\|\mathbf u=tcp
\]
Thus for sufficiently small $t$, 
 \[ f(\mathbf 0) + t\mathbf{d}\in 
 \mathcal{B}( f(t\mathbf u), tcp\|\mathbf d\|)=
 \mathcal{B}(g(t\mathbf u, \mathbf 0), tcp).\]
 Hence there exist $(\widehat{\mathbf{x}},  \widehat{\mathbf{y}}) \in \mathcal{B}((t\mathbf u, \mathbf 0), tp) \subseteq \mathcal{B}((t\mathbf u, \mathbf 0), t\tau/2)$
 with 
$ f(\mathbf 0)+ t\mathbf d= g(\widehat{\mathbf x},\widehat{\mathbf y})$.
Thus
 \[ f(\mathbf 0) + t\mathbf d-\left[ \begin{array}{c} \widehat{\mathbf y}\\ \mathbf 0\end{array} \right] =f(\widehat{\mathbf x}) \in \mbox{Im}(f).
 \] 
As $(\widehat{\mathbf x}, \widehat{\mathbf y}) \in \mathcal{B}((t\mathbf u, \mathbf 0), t\tau/2)$, each entry of $\widehat{\mathbf y}$ has magnitude at most $t\tau/2$.
The definition of $\tau$ now implies that 
the support of $f(\mathbf 0) + t\mathbf d + \widehat{\mathbf{y}}$ is the union of the supports of $f(\mathbf 0 )$ and $\mathbf d$. 
Condition (\ref{cond'}) follows from (\ref{cond})
the invertibility of a linear transformation is preserved by sufficiently small perturbations. 
\end{proof}

\noindent
{\bf Proof of the Matrix Liberation Theorem for the SSVP}

Assume that $A$ is an $m\times n$ matrix and 
$D$ is a nonzero matrix in $\mbox{Tan}_A^{\Sigma}$, $S$ is the pattern of $A$ in the direction of $D$, and $A$ has the SSVD with respect to $S$. 

We apply \Cref{IVTLiberate} to the function $f(K,L)=e^KAe^L$, and observe that the hypothesis 
(\ref{cond}) is equivalent to $\mbox{Tan}_f^{\Sigma} + \mbox{Span}(\mathcal{Z}(S))= \mathbb{R}^{m\times n}$.

\bibliographystyle{elsarticle-harv} 
\bibliography{Refs.bib}
\end{document}